\documentclass[11pt]{amsart}
\usepackage[pdftex,margin=1.1in, marginpar=.8in]{geometry}

\usepackage{amsmath}
\usepackage{amsfonts}
\usepackage{amssymb}
\usepackage{amsthm}

\usepackage[colorlinks=true,urlcolor=blue,linkcolor=blue,citecolor=blue,pagebackref,breaklinks]{hyperref}  
\usepackage{cite}

\usepackage{mathrsfs}
\usepackage{bm}
\usepackage{dsfont}
\usepackage{upgreek}
\usepackage{xcolor}
\usepackage{bbold}

\usepackage{fullpage}

\usepackage{graphicx}    
\usepackage[all,cmtip]{xy}
\usepackage{blkarray}
\usepackage{array}
\usepackage{todonotes}
\usepackage{tikz-cd}

\usepackage{enumitem} 

\usepackage{leftidx}
\usepackage{mathdots}
\usepackage{arydshln}
\usepackage{stmaryrd}  

\newenvironment{bsm}
{\left[\begin{smallmatrix}}
{\end{smallmatrix}\right]}

\numberwithin{equation}{subsection}

\newtheorem{thm}{Theorem}[subsection]
\newtheorem{prop}[thm]{Proposition}

\theoremstyle{remark}
\newtheorem{rmk}[thm]{Remark}

\newcommand{\bA}{\mathbb{A}}

\newcommand{\bC}{\mathbb{C}}

\newcommand{\bF}{\mathbb{F}}

\newcommand{\bI}{\mathbb{I}}

\newcommand{\bQ}{\mathbb{Q}}
\newcommand{\bR}{\mathbb{R}}

\newcommand{\bU}{\mathbb{U}}

\newcommand{\bZ}{\mathbb{Z}}

\newcommand{\cF}{\mathcal{F}}

\newcommand{\cI}{\mathcal{I}}

\newcommand{\cK}{\mathcal{K}}

\newcommand{\cO}{\mathcal{O}}
\newcommand{\cP}{\mathcal{P}}

\newcommand{\fm}{\mathfrak{m}}

\newcommand{\fp}{\mathfrak{p}}

\newcommand{\fC}{\mathfrak{C}}

\newcommand{\sB}{\mathscr{B}}
\newcommand{\sC}{\mathscr{C}}

\newcommand{\sG}{\mathscr{G}}

\newcommand{\sT}{\mathscr{T}}

\DeclareMathAlphabet{\mathpzc}{OT1}{pzc}{m}{it}

\newcommand{\pzC}{\mathpzc{C}}

\newcommand{\pzM}{\mathpzc{M}}

\newcommand{\pzV}{\mathpzc{V}}

\DeclareMathOperator{\GL}{GL}

\DeclareMathOperator{\U}{U}
\DeclareMathOperator{\OO}{O}
\DeclareMathOperator{\SL}{SL}
\DeclareMathOperator{\Sp}{Sp}

\DeclareMathOperator{\Mp}{Mp}

\newcommand{\diag}{\mathrm{diag}}

\newcommand{\Meas}{\mathpzc{M}eas}

\newcommand{\ord}{\mathrm{ord}}

\newcommand{\Tr}{\mathrm{Tr}}
\newcommand{\triv}{\mathrm{triv}}

\newcommand{\vol}{\mathrm{vol}}

\DeclareMathOperator{\Her}{Her}
\DeclareMathOperator{\Hom}{Hom}

\newcommand{\qexp}{q\text{-}\mathrm{exp}}

\newcommand{\lhra}{\ensuremath{\lhook\joinrel\longrightarrow}}
\newcommand{\lra}{\longrightarrow}
\newcommand{\ra}{\rightarrow}
\newcommand{\hra}{\hookrightarrow}

\newcommand{\ol}{\overline}

\newcommand{\llb}{\llbracket}
\newcommand{\rrb}{\rrbracket}

\newcommand{\bid}{\mathbf{1}}

\newdir{d}{{\subset}}   

\makeatletter
\newcommand*\bdot{\mathpalette\bdot@{.5}}
\newcommand*\bdot@[2]{\mathbin{\vcenter{\hbox{\scalebox{#2}{$\,\,\m@th#1\bullet\,\,$}}}}}
\makeatother  

\makeatletter
\def\l@section{\@tocline{1}{0pt}{1pc}{}{}}
\def\l@subsection{\@tocline{2}{0pt}{1pc}{4.6em}{}}
\def\l@subsubsection{\@tocline{3}{0pt}{1pc}{7.6em}{}}
\renewcommand{\tocsection}[3]{%
  \indentlabel{\@ifnotempty{#2}{\makebox[2.3em][l]{%
    \ignorespaces#1 #2.\hfill}}}#3}
\renewcommand{\tocsubsection}[3]{%
  \indentlabel{\@ifnotempty{#2}{\hspace*{2.3em}\makebox[2.3em][l]{%
    \ignorespaces#1 #2.\hfill}}}#3}
\renewcommand{\tocsubsubsection}[3]{%
  \indentlabel{\@ifnotempty{#2}{\hspace*{4.6em}\makebox[3em][l]{%
    \ignorespaces#1 #2.\hfill}}}#3}
\makeatother

\makeatletter
\renewcommand*\env@matrix[1][\arraystretch]{%
  \edef\arraystretch{#1}%
  \hskip -\arraycolsep
  \let\@ifnextchar\new@ifnextchar
  \array{*\c@MaxMatrixCols c}}
\makeatother

\newtheoremstyle{named}%
    {}{}{\itshape}{}{\bfseries}{.}{.5em}{\thmnote{#3}}
\theoremstyle{named}

\newcommand{\uchi}{\protect\underline{\chi}}

\newcommand{\supp}{\mathrm{supp}}
\newcommand{\ttt}{{\tt t}}
\newcommand{\bvarphi}{\bm{\varphi}}
\newcommand{\Pord}{P\text{-ord}}

\title{Test Schwartz functions at $p$ for theta lifts of Hida families}
\author{Zheng Liu}

\address{Z. L.:University of California, Santa Barbara, CA, United States}
\email{\href{mailto:zliu@math.ucsb.edu}{zliu@math.ucsb.edu}}
\date{}

\begin{document}

\maketitle

\begin{abstract}
We construct Hida families of theta lifts from definite orthogonal and unitary groups. A major ingredient of the construction is the choice of test Schwartz functions at places dividing $p$. We select a special type of Schwartz functions endowed with the equivariance property for the action of $\bU_p$-operators.
\end{abstract}

\tableofcontents 
\numberwithin{equation}{subsection}

\section{Introduction}

Hida families are families of ordinary automorphic forms parameterized by finite covers of the spectrums of Iwasawa algebras. Since being introduced by Hida in the 1980's \cite{HidaFamily}, they have found broad applications in studying arithmetic properties of automorphic forms and Galois representations. When implementing Hida families in those  applications, one aspect that is often essential is to construct certain explicit Hida families. For $\GL_2$, examples of explicit Hida families include Eisenstein families and CM families. They are used in the proofs of  Iwasawa Main Conjectures for totally real fields \cite{WiMC} and CM fields \cite{HiTi93,HiTi94}. They are also a crucial ingredient in the construction of various $p$-adic $L$-functions, including  Rankin--Selberg $p$-adic $L$-functions \cite{Hi88,HiFRankin} and anti-cyclotomic $p$-adic $L$-functions for modular forms and their Rankin--Selberg products. Explicit Eisenstein families on groups of higher ranks are constructed and utilized in \cite{SU,HsiehMC,WanU31,CLW,WanL,EHLS,SLF,LRtz,KEF}. 
In this paper, we turn to explicit CM families, {\it i.e.} families of theta lifts, on groups of higher ranks. 

Let $p$ be an odd integer. Like  Eisenstein families, one subtle layer of the construction of $p$-adic families of theta lifts  is the selection of test sections, in particular at places dividing $p$. Our first result is about a special type of test Schwartz functions at $p$ for theta lifts between symplectic and split orthogonal groups and between quasi-split unitary groups. We show that they have nice equivariant properties with respect to $\bU_p$-operators.

Let $F_\fp$ be a $p$-adic field and $K_\fp=F_\fp$ or $F_\fp\oplus F_\fp$ or a quadratic field extension of $F_\fp$. Let $\cO_\fp$ be the valuation ring of $K_\fp$. (When $K_\fp=F_\fp\oplus F_p$, $\cO_\fp$ is the direct sum of valuation rings of the two copies of $F_\fp$.) Fix positive integers $n\geq m$ and denote by $G_\fp,H_\fp$ the symplectic group of degree $2n$  and the split orthogonal group of degree $2m$ over $K_\fp=F_\fp$ or the quasi-split unitary groups of degree $2n$ and degree $2m$ with respect to the $K_\fp/F_\fp$, $[K_\fp:F_\fp]=2$. More precisely, 
\begin{align*}
  G_\fp&=\left\{g\in\GL_{2n}(K_\fp):\ltrans{\bar{g}} \begin{bmatrix}\phantom{0}&\bid_n\\ -\bid_n&\end{bmatrix} g=\begin{bmatrix}\phantom{0}&\bid_n\\ -\bid_n\end{bmatrix}\right\},\\
  H_\fp&=\left\{g\in\GL_{2m}(K_\fp):\ltrans{\bar{g}} \begin{bmatrix}\phantom{0}&\bid_m\\ \bid_m&\end{bmatrix} g= \begin{bmatrix}\phantom{0}&\bid_m\\ \bid_m\end{bmatrix}\right\}.
\end{align*}
$G_\fp\times H_\fp$ acts on the space of Schwartz functions on $M_{2m,n}(K_\fp)$ via Weil representation. For 
\begin{align}
   \ttt&= \diag(a_1,\dots,a_n,\bar{a}^{-1}_1,\dots,\bar{a}^{-1}_n)\in G_\fp, &&a_ja^{-1}_{j+1}\in \cO_\fp, \,1\leq j\leq n-1, \,a_n\bar{a}_n\in \cO_\fp,\\
   \label{eq:T+H} \ttt'&= \diag(a_1,\dots,a_m,\bar{a}^{-1}_1,\dots,\bar{a}^{-1}_m)\in H_\fp, &&a_ja^{-1}_{j+1}\in \cO_\fp, \,1\leq j\leq m-1, \,a_m\bar{a}_m\in \cO_\fp,
\end{align} we define the $\bU_\fp$-operators
\begin{align}
  \label{eq:UGpt} U^{G_\fp}_{\fp,\ttt}&=\int_{U_{G_\fp,\cO_\fp}}  (\text{action by $u\ttt$}) \, du 
   = \vol\left(\ttt U_{G_\fp,\cO_\fp}\ttt^{-1} \right)\hspace{-2em} \sum_{u\in U_{G_\fp,\cO_\fp}/\ttt U_{G_\fp,\cO_\fp}\ttt^{-1}} \text{action by $u\ttt$},\\
    \label{eq:UHpt} U^{H_\fp}_{\fp,\ttt'}&=\int_{U_{H_\fp,\cO_\fp}}  (\text{action by $u\ttt'$}) \, du 
   = \vol\left(\ttt' U_{H_\fp,\cO_\fp}\ttt^{\prime-1} \right)\hspace{-2em} \sum_{u\in U_{H_\fp,\cO_\fp}/\ttt' U_{H_\fp,\cO_\fp}\ttt^{\prime-1}} \text{action by $u\ttt$},
\end{align}
where $U_{G_\fp,\cO_\fp}$ (resp. $U_{H_\fp,\cO_\fp}$) denotes the $\cO_\fp$ points of the unipotent subgroup of the standard Borel subgroup of $G_\fp$ (resp. $H_\fp$) (see \S\ref{sec:UpSch-notation} for the precise definition).

Given an $m$-tuple of characters $\uchi_\fp=(\chi_{\fp,1},\dots,\chi_{\fp,m})$ of $\cO^\times_\fp$,  writing $X\in M_{2m,n}(K_\fp)$ in blocks as $X=\begin{blockarray}{ccc}m&n-m\\ \begin{block}{[cc]c}X_1&X_2&m\\X_3&X_4&m \\ \end{block}\end{blockarray}$, we define the Schwartz function $\phi_{\uchi_\fp}$ on $M_{2m,n}(F_\fp)$  as
\begin{equation}\label{eq:Schwchi}
\begin{aligned}
   \phi_{\uchi_\fp}\left(X\right)
   =&\, \prod_{j=1}^m
   \chi_{\fp,j}\left(D_j(X_1) \right)\cdot \mathds{1}_{B^-_m(\cO_{\fp})B_m(\cO_{\fp})}(X_1)
   \cdot \mathds{1}_{M_{m,n-m}(\cO_{\fp})}(X_2)
   \cdot \mathds{1}_{M_{m,n}(\cO_{\fp})}(X_3,X_4).
\end{aligned},
\end{equation}
with $D_j(X_1)$ the determinant of the upper-left $j\times j$ block of $X_1$ and $B_m(\cO_\fp)$ (resp. $B^-_m(\cO_\fp)$) the subgroup of upper triangular (resp. lower triangular) matrices in $\GL_m(\cO_\fp)$.

\begin{thm}\label{thm:Up-equiv}
The Schwartz function $\phi_{\uchi_\fp}$ defined in \eqref{eq:Schwchi} satisfies the following equivariance property for the $U_p$-operators defined in \eqref{eq:UGpt}\eqref{eq:UHpt}: For $\ttt= \diag(a_1,\dots,a_m,\bar{a}^{-1}_1,\dots,\bar{a}^{-1}_m)\in H_\fp$ satisfying the conditions in \eqref{eq:T+H} and $\breve{\ttt}=\diag(a_1,\dots,a_m,1,\dots,1,\bar{a}^{-1}_1,\dots,\bar{a}^{-1}_m1,\dots,1)\in G_\fp$,
\begin{align*}
   &\vol\left(\ttt U_{H_\fp,\cO_\fp}\ttt^{-1}\right)^{-1}\cdot U^{H_\fp}_{\fp,\ttt}\phi_{\uchi_\fp} 
   =  |\det A(\breve{\ttt})|^{m}_\fp\,\vol\left(A(\breve{\ttt}) U_{\GL_n,\cO_\fp}A(\breve{\ttt})^{-1}\right)^{-1}\cdot U^{G_\fp}_{\fp,\breve{\ttt}}\phi_{\uchi_\fp}.
\end{align*}
Here $A(\breve{\ttt})=\diag(a_1,\dots,a_m,1,\dots,1)\in\GL_n(K_\fp)$, $U_{\GL_n,\cO_\fp}$ is the subgroup of unipotent upper triangular matrices in $\GL_n(\cO_\fp)$, and we use Haar measures on $U_{H_\fp,\cO_\fp},U_{G_\fp,\cO_\fp},U_{\GL_n,\cO_\fp}$ with total volumes equal to $1$. 
\end{thm}
This is proved in Theorem~\ref{thm:local}, where more cases with $H_\fp$ of odd degree and $\breve{\ttt}$ of a slightly more general shape are covered. \\

The $\bU_p$-equivariance property of the Schwartz function $\phi_{\uchi_\fp}$ makes it convenient for constructing theta lifts of Hida families. We consider such a construction in a simplest case for theta lifts from a compact orthogonal group of even degree to a symplectic group or from a compact unitary group to a quasi-split unitary group. Let $F$ be a totally real field and $K$ be $F$ or a quadratic CM extension of $F$ in which $p$ is unramified.

Let $G'$ be the algebraic group over $\bQ$ such that for a $\bQ$-algebra $R$, 
\[
    G'(R)=\left\{g\in \GL_{2n}(R\otimes_\bQ K):\ltrans{\bar{g}}\begin{bmatrix}&\bid_n\\-\bid_n\end{bmatrix}=\nu(g)\begin{bmatrix}&\bid_n\\-\bid_n\end{bmatrix},\,\nu(g)\in R^\times\right\}
\]
Let $W$ be a symmetric space over $K$ when $K=F$ (resp. Hermitian space over $K$ when $K$ is a CM extension of $F$) of dimension $2m$, positive definite at all archimedean places and split at all places above $p$. Denote by $H$ the subgroup of ${\rm Res}_{K/\bQ}\GL_{n,K}$  preserving the symmetric (resp. Hermtian) form on $W$. We fix an isomorphism between $W\otimes_{\bQ}\bQ_p$ and $(K\otimes_\bQ\bQ_p)^{\oplus 2m}$ equipped with the symmetric (resp. Hermtian) form $\begin{bsm}&\bid_m\\ \bid_m\end{bsm}$. Denote by $P\subset\GL_n$ the parabolic subgroup consisting of elements with $0$ lower-left $(n-m)\times m$ blocks. (For this parabolic subgroup $P$, we can define $P$-ordinary projector for $p$-adic automorphic forms on $G'$ and have Hida theory established for them. See \S\S\ref{sec:Pord},\ref{sec:HidaTheory}.)

\begin{thm} 
Suppose that we are in one of the cases listed at the end of  \S\ref{sec:HidaTheory}. Given an $\bI$-adic Hida family $\bvarphi$ on $H$ of tame level $K^{H,p}_f$ (as in \S\ref{sec:globalsetting}), fixing a Schwartz function $\phi^p_f$ on $W^n\otimes_F\bA^p_{F,f}$ taking values inside $\cO_L\cap \ol{\bQ}$, invariant under an open compact subgroup $K^{G',p}_f$ of $\prod_{v\neq p \text{ finite place of $\bQ$}}G'(\bZ_v)$, there exists an $\bI$-adic $P$-ordinary family
\[
    {\bm \theta}(\bvarphi)\in  \pzM^{G',n-m}_{\Pord}(K^{G',p}_f,\cO_L)\otimes_{\Lambda_P} \bI_{\bvarphi},
\]
(see \S\ref{sec:HidaTheory} for the exact definition of the space on the right hand side), interpolating the theta lifts of specializations of $\bvarphi$ at finite-order points. More precisely, for all points $\lambda:\bI_{\bvarphi}\ra\ol{\bQ}_p$ such that the composition $\prod_{\fp\mid p}(\cO^\times_\fp)^m\cong T_{H,p}(\bZ_p)\ra \cO_L\llb T_{H,p}(\bZ_p)\rrb^\times\ra \bI^\times_{\bvarphi}\stackrel{\lambda}{\ra} \bar{\bQ}^\times_p$ at each $\fp$ is an $m$-tuple of finite-order characters $\uchi_\fp$,
\begin{equation}\label{eq:theta-dagger}
    \lambda({\bm \theta}(\bvarphi))
    =\vol(K^{H,p}_f)^{-1}\cdot\imath_{G,G'}\left(\theta^\dagger\left(\phi^p_f\otimes \prod\nolimits_{\fp\mid p}\phi_{\uchi_\fp}\otimes\prod\nolimits_{\sigma:F\ra\bR} \phi^\sG_\sigma,\lambda(\bvarphi)\right) \right),
\end{equation}
where $\phi_{\uchi_\fp}$ is the Schwartz function defined in \eqref{eq:Schwchi}, $\phi^\sG_\sigma$ is the Gaussian function defined in \eqref{eq:Gaussian},  $\theta^\dagger$ denotes the theta lift from $H$ to $G=\{g\in G:\nu(g)=1\}$ with a modification at $p$ defined in \eqref{eq:thetadagger}, and $\imath_{G,G'}$ is the extension by zero for the embedding \eqref{eq:imathembd}.
\end{thm}

To construct ${\bm \theta}(\bvarphi)$, we first construct an $\cO_L\llb T_{H,p}(\bZ_p)\rrb$-adic family of theta series, as $q$-expansions for $G'$ with coefficients inside the space \eqref{eq:MeasH} (which can be viewed as $\cO_L\llb T_{H,p}(\bZ_p)\rrb$-adic forms on $H$), associated to the Schwartz functions  $\phi^p_f\otimes \prod\nolimits_{\fp\mid p}\phi_{\uchi_\fp}\otimes\prod\nolimits_{\sigma:F\ra\bR} \phi^\sG_\sigma$. Then  we formulate a $p$-adic Petersson inner product on $H$ and apply it to the $\cO_L\llb T_{H,p}(\bZ_p)\rrb$-adic family of theta series and the given Hida family $\bvarphi$ on $H$. This produces an  $\cO_L\llb T_{H,p}(\bZ_p)\rrb$-adic family of $q$-expansions on $G'$. Then in the cases  where suitable Hida theory is available form \cite{HidaCon,PiHida,LRtz,Marcil},  we combine the $\bU_p$-equivariance of  $\phi_{\uchi_\fp}$ proved in Theorem~\ref{thm:Up-equiv} and Hida theory for $P$-ordinary forms to show that the constructed $\cO_L\llb T_{H,p}(\bZ_p)\rrb$-adic family of $q$-expansions on $G'$ come from a $P$-ordinary Hida family.

\vspace{1em}
\noindent{\bf Acknowledgments.} I would like to thank Ming-Lun Hsieh and Xiaoyu Zhang for helpful discussions. During the preparation of this paper, the  author was partially supported by the NSF grants DMS-2001527 and DMS-2501507.

\section{$\bU_\fp$-equivariant test Schwartz functions at $\fp$}\label{sec:UpSch}

\subsection{The setting and notation}\label{sec:UpSch-notation}
Let $p$ be an odd integer and $F_\fp$ be a finite extension of $\bQ_p$. Put $K_\fp=F_\fp$ or $F_\fp\oplus F_\fp$ or a quadratic field extension of $F_\fp$, and $\epsilon=0$ or $1$. For $a\in K_\fp$, we use $\bar{a}$ to denotes $a$ itself when $K_\fp=F_\fp$ and the image of $a$ under the nontrivial element in ${\rm Aut}(K_\fp/F_\fp)$ when $[K_\fp:F_\fp]=2$.

Fix positive integers $n\geq m$. With $K_\fp=F_\fp$ (resp. $[K_\fp:F_\fp]=2$), let 
\begin{align*}
   \Sp_{2n,\fp} \text{ (resp. $\U_{n,n,\fp}$)}&= \left\{g\in\GL_{2n}(K_\fp):\ltrans{\bar{g}} \begin{bmatrix}\phantom{0}&\bid_n\\ -\bid_n&\end{bmatrix} g=\begin{bmatrix}\phantom{0}&\bid_n\\ -\bid_n\end{bmatrix}\right\},
\end{align*}
and
\begin{align*}
  \OO_{2m+\epsilon,\fp}\text{ (resp. $\U_{2m+\epsilon,\fp}$)}
  &=\left\{g\in\GL_{2m}(K_\fp):\ltrans{\bar{g}} J_{2m+\epsilon} g= \begin{bmatrix}\phantom{0}&\bid_m\\ \bid_m\end{bmatrix}\right\}, 
\end{align*}
with 
\begin{equation}\label{eq:J}
    J_{2m+\epsilon}=\left\{\begin{array}{ll}
    \begin{bmatrix}\phantom{0}&\bid_m\\ \bid_m&\end{bmatrix}, & \epsilon=0, \\[2ex]
    \begin{bmatrix}\phantom{0}&\bid_m\\ \bid_m&\\&&1\end{bmatrix}, &\epsilon=1,
    \end{array}\right. 
\end{equation}
Let $\Mp_{2n,\fp}$ be the metaplectic double-cover of $\Sp_{2n,\fp}$, and we write it as $\Sp_{2n,\fp}\times\{\pm 1\}$ as a set.

Fix an additive character $\psi_\fp:F_\fp\ra\bC^\times$ with $\ker\psi_\fp=\cO_{F_\fp}$, the valuation of $F_\fp$.  We consider theta correspondence for the dual pair $(G_\fp, H_\fp)$, in the following four cases:
\begin{align*}
   (G_\fp,H_\fp)=\left\{\begin{array}{lll}  
   (\Sp_{2n,\fp},\OO_{2m,\fp}), & (K_\fp=F_\fp ,\epsilon=0),&\hspace{4em} \text{Case (OS0)},\\
   (\Mp_{2n,\fp},\OO_{2m+1,\fp}), & (K_\fp=F_\fp,\epsilon=1), & \hspace{4em}\text{Case (OS1)},\\
   (\U_{n,n,\fp},\U_{2m,\fp}), & ([K_\fp:F_\fp]=2,\epsilon=0), & \hspace{4em}\text{Case (U0)},\\
   (\U_{n,n,\fp},\U_{2m+1,\fp}), & ([K_\fp:F_\fp]=2,\epsilon=1), & \hspace{4em}\text{Case (U1)}.
   \end{array}\right.
\end{align*}

\vspace{.5em}

We introduce some notations. We denote by $\cO_\fp$ to denote the valuation ring of $K_\fp$. (When $K_\fp=F_\fp\oplus F_\fp$, $\cO_\fp=\cO_{F_\fp}\oplus \cO_{F_\fp}$.) Denote by $T_{G_\fp}\subset G_\fp$ and $T_{H_\fp}\subset H_\fp$ their maximal tori consisting of all diagonal matrices. (In case (OS1), we define $T_{G_\fp}$ to be the inverse image of the maximal torus of diagonal matrices in $\Sp_{2n,\fp}$ under the projection $\Mp_{2n,\fp}\ra \Sp_{2n,\fp}$.) For any positive integer $r$, let $B_{\GL_r,\fp}\subset \GL_r(K_\fp)$ be the standard Borel subgroup consisting of upper triangular matrices and $U_{\GL_r,\fp}\subset B_{\GL_r,\fp}$ the unipotent subgroup. Fix the following Borel subgroups of $G_\fp,H_\fp$: 
\begin{align*}
   B_{G_\fp}&=\left\{\begin{array}{ll}\left\{\begin{bsm} \ast&\cdots &\ast&\ast&\cdots&\ast\\ 
   &\ddots&\vdots&\vdots&\ddots&\vdots\\&&\ast&\ast&\cdots&\ast\\ 
   &&&\ast\\
   &&&\vdots&\ddots\\
   &&&\ast&\cdots&\ast\end{bsm}\in G_\fp\right\}, 
   & \text{Cases (OS0)(U0)(U1)},\\[5ex]
   \left\{(g,\pm 1):g=\begin{bsm} \ast&\cdots &\ast&\ast&\cdots&\ast\\ 
   &\ddots&\vdots&\vdots&\ddots&\vdots\\&&\ast&\ast&\cdots&\ast\\ 
   &&&\ast\\
   &&&\vdots&\ddots\\
   &&&\ast&\cdots&\ast\end{bsm}\in \Sp_{2n,\fp}\right\}, &\text{Case (OS1)},
   \end{array}\right.\\
   B_{H_\fp}&=\left\{\begin{array}{ll}\left\{\begin{bsm} \ast&\cdots &\ast&\ast&\cdots&\ast\\ 
   &\ddots&\vdots&\vdots&\ddots&\vdots\\&&\ast&\ast&\cdots&\ast\\ 
   &&&\ast\\
   &&&\vdots&\ddots\\
   &&&\ast&\cdots&\ast\end{bsm}\in H_\fp\right\}, & \text{Case (OS0)(U0)},\\[5ex]
   \left\{\begin{bsm} \ast&\cdots &\ast&\ast&\cdots&\ast&\ast\\
    &\ddots&\vdots&\vdots&\ddots&\vdots&\vdots\\
    &&\ast&\ast&\cdots&\ast&\ast\\ 
    &&&\ast \\ 
    &&&\vdots&\ddots\\&&&\ast&\cdots&\ast\\ 
    &&&\ast&\cdots&\ast&\ast\end{bsm}\in H_\fp\right\}, & \text{Case (OS1)(U1)},\end{array}\right.
\end{align*}
and denote by $U_{G_\fp},U_{H_\fp}$ their unipotent subgroups. Put
\begin{align*}
   B_{\GL_r,\cO_p}&=B_{\GL_r,\fp}\cap\GL_r(\cO_\fp),
   &B_{G_\fp,\cO_\fp}&= B_{G_\fp}\cap  \GL_{2n}(\cO_\fp),
   &B_{H_\fp,\cO_\fp}&= B_{H_\fp}\cap  \GL_{2n}(\cO_\fp),\\
   U_{\GL_r,\cO_p}&=U_{\GL_r,\fp}\cap\GL_r(\cO_\fp),
   &U_{G_\fp,\cO_\fp}&= U_{G_\fp}\cap  \GL_{2n}(\cO_\fp),
   &U_{H_\fp,\cO_\fp}&= U_{H_\fp}\cap  \GL_{2m+\epsilon}(\cO_\fp).
\end{align*}
(In case (OS1), $B_{G_\fp,\cO_\fp}=B_{G_\fp}\cap\left( \GL_{2n}(\cO_p)\times\{\pm 1\}\right)$, $U_{G_\fp,\cO_\fp}=U_{G_\fp}\cap \left(\GL_{2n}(\cO_p)\times\{1\}\right)$.) Also, we let $B^-_{\ast\ast}$, $U^-_{\ast\ast}$ be the corresponding opposite Borel and unipotent subgroups, {\it i.e.} consisting of elements whose transpose belong to $B_{\ast\ast}$, $U_{\ast\ast}$.  We fix Haar measures on $U_{H_\fp,\cO_\fp},U_{G_\fp,\cO_\fp},U_{\GL_n,\cO_\fp}$ with total volumes equal to $1$.

\subsection{Some formulas for Weil representations}\label{sec:WeilRep}
We recall the formulas for the action of $H_\fp$ and the standard Siegel parabolic subgroup of $G_\fp$ in the Schr\"{o}dinger model of Weil representation $\omega_{\psi_\fp}$. Let $M_{2m+\epsilon,n}(K_\fp)$ be the set of $(2m+\epsilon)$-by-$n$ matrices with entries in $K_\fp$. For a Schwartz function $\phi$ on $M_{2m+\epsilon,n}(K_\fp)$,  $h\in H_\fp$ and $\begin{bmatrix}A&B\\0&\ltrans{\bar{A}}^{-1}\end{bmatrix}\in G_\fp$ (and $\left(\begin{bmatrix}A&B\\0&\ltrans{\bar{A}}^{-1}\end{bmatrix},\pm 1\right)\in G_\fp$ in Case (OS1)), we have
\[
   \omega_{\psi_\fp}(\bid_{G_\fp},h)\phi(X) =\phi(\ltrans{hX}),
\]
and
\begin{align*}
   &\omega_{\psi_\fp}\left(\begin{bmatrix}A&B\\0&\ltrans{\bar{A}}^{-1}\end{bmatrix},\bid_{H_\fp}\right)\phi(X)
   && \text{Case (OS0)(U0)(U1)}\\
   =&\,\psi_\fp\left(\frac{1}{2}\Tr B\ltrans{\bar{A}} \ltrans{\bar{X}} J_{2m+\epsilon} X\right)
   \cdot \tilde{\chi}^\epsilon_{K_\fp/F_\fp}(\det A) |\det A|^{\frac{2m+\epsilon}{2}}_\fp \phi(XA),\\
   &\omega_{\psi_\fp}\left(\left(\begin{bmatrix}A&B\\0&\ltrans{A}^{-1}\end{bmatrix},\pm 1\right),\bid_{H_\fp}\right)\phi (X)
   &&\text{Case (OS1)}\\
   =&\,\pm\psi_\fp\left(\frac{1}{2}\Tr B\ltrans{\bar{A}} \ltrans{\bar{X}} J_{2m+1} X\right)
   \cdot  \gamma(\det A,\frac{1}{2}\psi_\fp)^{-1}|\det A|^{\frac{2m+1}{2}}_\fp \phi(XA), 
\end{align*}
with
\[
    \tilde{\chi}_{K_\fp/F_\fp}=\left\{\begin{array}{ll} \triv, &\text{Case (OS0)(OS1)},\\[1ex] \text{a character of $K^\times_\fp$ extending $\chi_{K_\fp/F_\fp}:F^\times_\fp\ra\bC^\times$},& \text{Case (U0)(U1)},\end{array}\right.
\]
and $\gamma(\det A,\frac{1}{2}\psi_\fp)$ is defined as in \cite[Proposition~A.11]{Rao}.



\subsection{The $\bU_p$-operators}\label{sec:Up-oper}

We are interested in selecting test Schwartz functions on $M_{2m+\epsilon,n}(K_\fp)$ possessing nice properties with respect to the action of $\bU_p$-operators so that they are well-suited for constructing $p$-adic families of theta lifts.

We define $\bU_p$-operators indexed by the following semigroups
\begin{align}
   \label{eq:T+GH}  T^+_{G_\fp}&=\left\{\ttt\in T_{G_\fp}:\ttt U_{G_\fp,\cO_\fp}\ttt^{-1}\subset U_{G_\fp,\cO_\fp}\right\},
   &T^+_{H_\fp}&=\left\{\ttt\in T_{H_\fp}:\ttt U_{H_\fp,\cO_\fp}\ttt^{-1}\subset U_{H_\fp,\cO_\fp}\right\}.
\end{align}
By the definition, given
\begin{align*}
   \diag(a_1,\dots,a_n,\bar{a}^{-1}_1,\dots,\bar{a}^{-1}_n)&\in T_{G_\fp}, &&\text{Case (OS0)(U0)(U1)}\\
   \left(\diag(a_1,\dots,a_n,\bar{a}^{-1}_1,\dots,\bar{a}^{-1}_n),\pm 1\right)&\in T_{G_\fp}, &&\text{Case (OS1)}
\end{align*}
it belongs to $T^+_{G\fp}$ iff $a_ja^{-1}_{j+1}\in \cO_\fp$, $1\leq j\leq n-1$, and $a_n\bar{a}_n\in \cO_\fp$, and given 
\begin{align*}
   \diag(a_1,\dots,a_m ,\bar{a}^{-1}_1,\dots,\bar{a}^{-1}_m)&  \in T_{H_\fp}, &&\text{Case (OS0)(U0)}\\
   \diag((a_1,\dots,a_m ,\bar{a}^{-1}_1,\dots,\bar{a}^{-1}_m,u)&  \in T_{H_\fp} , &&\text{Case (OS1)(U1)} 
\end{align*}
it belongs to $T^+_{H_\fp}$ iff $a_ja^{-1}_{j+1}\in \cO_\fp$, $1\leq j\leq m-1$, and $a_m\bar{a}_m,a_mu^{-1}\in \cO_\fp$.

Given $\ttt\in T^+_{G_\fp}$ we define the (unnormalized) $\bU_p$-operator $U^{G_\fp}_{\fp,\ttt}$ as
\begin{align}
   \label{eq:UGp} U^{G_\fp}_{\fp,\ttt}=\int_{U_{G_\fp,\cO_\fp}}  (\text{action by $u\ttt$}) \, du 
   = \vol\left(\ttt U_{G_\fp,\cO_\fp}\ttt^{-1} \right)\hspace{-2em} \sum_{u\in U_{G_\fp,\cO_\fp}/\ttt U_{G_\fp,\cO_\fp}\ttt^{-1}} \text{action by $u\ttt$}.
\end{align}
Similarly, for $\ttt\in T^+_{H_\fp}$, define
\begin{align}
   \label{eq:UHp} U^{H_\fp}_{\fp,\ttt}=\int_{U_{H_\fp,\cO_\fp}}  (\text{action by $u\ttt$}) \, du 
   = \vol\left(\ttt U_{H_\fp,\cO_\fp}\ttt^{-1} \right)\hspace{-2em} \sum_{u\in U_{H_\fp,\cO_\fp}/\ttt U_{H_\fp,\cO_\fp}\ttt^{-1}} \text{action by $u\ttt$}.
\end{align}

\subsection{A special type of test Schwartz functions}\label{sec:Schwartz}
Given an $m$-tuple $\uchi_\fp=(\chi_{\fp,1},\chi_{\fp,2},\dots,\chi_{\fp,m})$ of finite-order characters of $\cO^\times_\fp$, define the continuous function $\alpha_{\uchi_\fp}$ on $M_{m,m}(\cK_\fp)$  as
\begin{align*}
   \alpha_{\uchi_\fp}\begin{pmatrix}
   x_{11}&x_{12}&\cdots&x_{1m}\\
   x_{21}&x_{22}&\cdots&x_{2m}\\
   \vdots&\vdots&\ddots&\vdots\\
   x_{m1}&x_{m2}&\cdots&x_{mm}
   \end{pmatrix}
   &= \prod_{j=1}^m 
   \chi_{\fp,j}\left(\det\begin{pmatrix}
   x_{11}&\cdots&x_{1j}\\
   \vdots&\ddots&\vdots\\
   x_{j1}&\cdots&x_{jj}\end{pmatrix} \right),
\end{align*}
and the Schwartz function $\phi_{\uchi_\fp}$ on $M_{2m,n+\epsilon}(K_\fp)$ as

\begin{align*}
   \phi_{\uchi_\fp}\left(\begin{blockarray}{ccc}m&n-m\\ \begin{block}{[cc]c}X_1&X_2&m\\X_3&X_4&m+\epsilon\\ \end{block}\end{blockarray}\right)
   =&\, \alpha_{\uchi_\fp}(X_1)
   \cdot \mathds{1}_{B^-_m(\cO_{\fp})B_m(\cO_{\fp})}(X_1)\\
   &\times \mathds{1}_{M_{m,n-m}(\cO_{\fp})}(X_2)\cdot \mathds{1}_{M_{m+\epsilon,n}(\cO_{\fp})}(X_3,X_4).
\end{align*}
In next section, we show that Schwartz functions of such shapes satisfy a nice equivariance property for the action of the $\bU_\fp$-operators for $G_\fp$ and $H_\fp$ defined in \S\ref{sec:Up-oper}.

\subsection{The $\bU_p$-equivariance}

\begin{thm}\label{thm:local}
For
\[
   \ttt=\left\{\begin{array}{ll}\diag(a_1,\dots,a_m ,\bar{a}^{-1}_1,\dots,\bar{a}^{-1}_m)\in T^+_{H_\fp}, &\textup{Case (OS0)(U0)}\\[1ex]
   \diag((a_1,\dots,a_m ,\bar{a}^{-1}_1,\dots,\bar{a}^{-1}_m,u)\in T^+_{H_\fp}, &\textup{Case (OS1)(U1)} \end{array}\right.
\]
and
\[
   \breve{\ttt}=\left\{\begin{array}{ll}\diag(a_1,\dots,a_m,d\dots,d,\bar{a}^{-1}_1,\dots,\bar{a}^{-1}_m,\bar{d}^{-1},\dots,\bar{d}^{-1})\in T^+_{G_\fp}, &\textup{Case (OS0)(U0)(U1)}\\[1ex]
   \left(\diag(a_1,\dots,a_m,d\dots,d,\bar{a}^{-1}_1,\dots,\bar{a}^{-1}_m,\bar{d}^{-1},\dots,\bar{d}^{-1}),1\right)\in T^+_{G_\fp}, &\textup{Case (OS1)}
   \end{array}\right.
\]
where for \textup{Case (U1}) with $n>m$ and $K_\fp=F_\fp\oplus F_\fp$, we further assume $du^{-1}\in \cO^\times_\fp$, the  Schwartz function $\phi_{\uchi_\fp}$  defined in \S\ref{sec:Schwartz} satisfies
\begin{align*}
   &\vol(\ttt U_{H_\fp,\cO_\fp}\ttt^{-1})^{-1}\cdot U^{H_\fp}_{\fp,\ttt}\phi_{\uchi_\fp} \\
   = &\,\left(\tilde{\chi}_{K_\fp/F_\fp}(\det A(\breve{\ttt})) \gamma(\det(A(\breve{\ttt}),\frac{1}{2}\psi_\fp)\right)^\epsilon
   \cdot |\det A(\breve{\ttt})|^{\frac{2m+\epsilon}{2}}_\fp\vol(A(\breve{\ttt}) U_{\GL_n,\cO_\fp}A(\breve{\ttt})^{-1})^{-1}\cdot U^{G_\fp}_{\fp,\breve{\ttt}}\phi_{\uchi_\fp},
\end{align*}
with $A(\breve{\ttt})=\diag(a_1,\dots,a_m,d,\dots,d)\in\GL_n(K_\fp)$.
\end{thm}

\begin{proof}
Let $V_{\supp}=\begin{bmatrix} B^-_{\GL_m,\cO_\fp}B_{\GL_m,\cO_\fp}& M_{n-m,m}(\cO_\fp)\\ M_{m+\epsilon,m}(\cO_\fp)& M_{m+\epsilon,n-m}(\cO_\fp)\end{bmatrix}\subset M_{2m+\epsilon,n}(K_\fp)$. 
By the definition of $U^{G_\fp}_{\fp,\breve{\ttt}},U^{H_\fp}_{\fp,\ttt}$ in \eqref{eq:UGp}\eqref{eq:UHp} and the formulas  for Weil representations recalled in \S\ref{sec:WeilRep}, we have
\begin{align*}
   \left(U^{H_\fp}_{\fp,\ttt}\phi_{\uchi_\fp}\right)(X)
   &=\int_{U_{H_\fp,\cO_\fp}} \phi_{\uchi_\fp}\left(\ltrans{\ttt} \ltrans{u} X\right)\,du\\
   &= \alpha_{\uchi_\fp}(\ttt X) \int_{U_{H_\fp,\cO_\fp}} \mathds{1}_{V_{\supp}}\left(\ltrans{\ttt}\ltrans{u} X\right)\,du 
\end{align*}
and
\begin{align*}
   \left(U^{G_\fp}_{\fp,\breve{\ttt}}\phi_{\uchi_\fp}\right)(X) 
   = &\,\left(\tilde{\chi}_{K_\fp/F_\fp}(\det A(\breve{\ttt})) \gamma(\det(A(\breve{\ttt}),\frac{1}{2}\psi_\fp)\right)^\epsilon
   \cdot |\det A(\breve{\ttt})|^{\frac{2m+\epsilon}{2}}_\fp\\
     &\times \int_{\Her_n(\cO_\fp)}\psi_\fp\left(\Tr\sigma\ltrans{\bar{X}}J_{2m+\epsilon}X\right)\,d\sigma\int_{U_{\GL_n,\cO_\fp}} \phi_{\uchi_\fp}\left(XuA(\breve{\ttt})\right)\,du\\
   = &\,\left(\tilde{\chi}_{K_\fp/F_\fp}(\det A(\breve{\ttt})) \gamma(\det(A(\breve{\ttt}),\frac{1}{2}\psi_\fp)\right)^\epsilon
   \cdot |\det A(\breve{\ttt})|^{\frac{2m+\epsilon}{2}}_\fp\\ 
   &\times \alpha_{\uchi_\fp}(\ttt X) \cdot \mathds{1}_{\Her_n(\cO_\fp)} \left(\ltrans{\bar{X}} J_{2m+\epsilon} X\right) \int_{U_{\GL_n,\cO_\fp}} \mathds{1}_{V_\supp}\left(XuA(\breve{\ttt})\right)\,du
\end{align*}
with $J_{2m+\epsilon}$ given in \eqref{eq:J} and $\Her_n(\cO_\fp)=\{\sigma\in M_{n,n}(\cO_\fp):\sigma=\ltrans{\bar{\sigma}}\}.$

Defining functions $\phi_{H_\fp,\ttt},\phi_{G_\fp,\breve{\ttt}}$ on $M_{2m+\epsilon,n}(K_\fp)$ as
\begin{align*}
   \phi_{H_\fp,\ttt}(X)&=\vol(\ttt U_{H_\fp,\cO_\fp}\ttt^{-1})^{-1}\cdot  \int_{U_{H_\fp,\cO_\fp}} \mathds{1}_{V_{\supp}}\left(\ttt\ltrans{u} X\right)\,du \\
    \phi_{G_\fp,\breve{\ttt}}(X)&=\vol(A(\breve{\ttt}) U_{\GL_n,\cO_\fp}A(\breve{\ttt})^{-1})^{-1}
    \cdot \mathds{1}_{\Her_n(\cO_\fp)} \left(\ltrans{\bar{X}} J_{2m+\epsilon} X\right) \int_{U_{\GL_n,\cO_\fp}} \mathds{1}_{V_\supp}\left(XuA(\breve{\ttt})\right)\,du,
\end{align*}
we reduce to showing that 
\[
   \phi_{H_\fp,\ttt}=\phi_{G_\fp,\breve{\ttt}}.
\]
We will first show that both sides equal to the characteristic functions of their supports, and then show that they have the same support. \\

\noindent\underline{$\phi_{H_\fp,\ttt}=\mathds{1}_{\supp\, \phi_{H_\fp,\ttt}}$, $\phi_{G_\fp,\breve{\ttt}}=\mathds{1}_{\supp\, \phi_{G_\fp,\breve{\ttt}}}$.} \\[-1ex]

Given  $X\in \supp\,\phi_{H_\fp,\ttt}$, taking $u_0\in U_{H_{\fp,\cO_p}}$ such that $Y=\ttt\ltrans{u}_0 X\in V_{\supp}$, we have
\begin{align*}
    \phi_{H_\fp,\ttt}(X)
    &=\vol(\ttt U_{H_\fp,\cO_\fp}\ttt^{-1})^{-1}\cdot \vol\left(\{u\in U_{H_\fp,\cO_\fp}:\ttt\ltrans{u} X\in V_{\supp}\}\right)\\
    &=\vol(\ttt U_{H_\fp,\cO_\fp}\ttt^{-1})^{-1}\cdot \vol\left(\{u\in U_{H_\fp,\cO_\fp}: \ttt\ltrans{(u_0^{-1}u)}\ttt^{-1}Y\in V_{\supp}\}\right).
\end{align*}
It is easy to see that for $u^-\in U^-_{H_\fp}$ and $Y\in V_\supp$, $u^-Y\in  V_\supp$ if and only if $u^-\in U^-_{H_\fp,\cO_\fp}$. Hence,
\begin{align*}
    \phi_{H_\fp,\ttt}(X)
    &=\vol(\ttt U_{H_\fp,\cO_\fp}\ttt^{-1})^{-1} \cdot \vol\left(\{u\in U_{H_\fp,\cO_\fp}: \ttt\ltrans{(u_0^{-1}u)}\ttt^{-1}\in U^-_{H_\fp,\cO_\fp}\}\right)\\
    &=\vol(\ttt U_{H_\fp,\cO_\fp}\ttt^{-1})^{-1} \cdot \vol\left(\ttt^{-1} U^-_{H_\fp,\cO_\fp}\ttt\right)=1.
\end{align*}
This shows that  $\phi_{H_\fp,\ttt}=\mathds{1}_{\supp\, \phi_{H_\fp,\ttt}}$. Similarly, given  
$X\in \supp\,\phi_{G_\fp,\breve{\ttt}}$, taking $u_0\in U_{\GL_n,\cO_\fp}$ such that $Y=Xu_0A(\breve{\ttt})\in  V_{\supp}$, we have
\begin{align*}
   \phi_{G_\fp,\breve{\ttt}}(X)
   &=\vol(A(\breve{\ttt}) U_{\GL_n,\cO_\fp}A(\breve{\ttt})^{-1})^{-1}
   \cdot \vol\left(\left\{u\in U_{\GL_n,\cO_\fp}: XuA(\breve{\ttt})\in  V_{\supp}\right\}\right)\\
    &=\vol(A(\breve{\ttt}) U_{\GL_n,\cO_\fp}A(\breve{\ttt})^{-1})^{-1}
    \cdot\vol\left(\{u\in U_{\GL_n,\cO_\fp}: YA(\breve{\ttt})^{-1}(u_0^{-1}u)A(\breve{\ttt})\in V_{\supp} \}\right)
\end{align*}
For $u\in U_{\GL_n,\cO_\fp}$, we can write
\begin{align}
   \label{eq:A-1uA}  A(\breve{\ttt})^{-1}u^{-1}_0uA(\breve{\ttt})
    &=\begin{bmatrix}w_1&w_2\\0&w_4\end{bmatrix},
    &w_1\in U_{\GL_m,\fp},\,w_2\in M_{m,n-m}(K_\fp),\,w_4\in U_{\GL_{n-m},\cO_\fp}.
\end{align}
Then $Y\in V_{\supp}$ and $Y\begin{bmatrix}w_1&w_2\\0&w_4\end{bmatrix}\in V_{\supp}$ if and only if $w_1\in U_{m,\cO_\fp},w_2\in M_{m,n-m}(\cO_\fp)$, {\it i.e.} if and only if $\eqref{eq:A-1uA}\in U_{\GL_n,\cO_\fp}$. Hence, 
\begin{align*}
   \phi_{G_\fp,\breve{\ttt}}(X)
   = \vol(A(\breve{\ttt}) U_{\GL_n,\cO_\fp}A(\breve{\ttt})^{-1})^{-1}
   \cdot \vol\left(\{u\in U_{\GL_n,\cO_\fp}: A(\breve{\ttt})^{-1}u_0^{-1}uA(\breve{\ttt})\in  U_{\GL_n,\cO_\fp}\}\right)=1,
\end{align*}
and $\phi_{G_\fp,\breve{\ttt}}=\mathds{1}_{\supp\, \phi_{G_\fp,\breve{\ttt}}}$.\\

\noindent\underline{$\supp\, \phi_{H_\fp,\ttt}= \supp\, \phi_{G_\fp,\breve{\ttt}}$.}\\[-1ex]

We have
\begin{align*}
   \supp\, \phi_{H_\fp,\ttt}
   &= U^-_{H_\fp,\cO_\fp} \, \ttt^{-1}\begin{bmatrix} B^-_{\GL_m,\cO_\fp}B_{\GL_m,\cO_\fp}& M_{n-m,m}(\cO_\fp)\\ M_{m+\epsilon,m}(\cO_\fp)& M_{m+\epsilon,n-m}(\cO_\fp)\end{bmatrix}\\
   &= U^-_{H_\fp,\cO_\fp} \, \ttt^{-1}\begin{bmatrix} B_{\GL_m,\cO_\fp}& M_{n-m,m}(\cO_\fp)\\ M_{m+\epsilon,m}(\cO_\fp)& M_{m+\epsilon,n-m}(\cO_\fp)\end{bmatrix}\\
   &= U^-_{H_\fp,\cO_\fp} \, \ttt^{-1}\begin{bmatrix} \bid_m& 0\\ M_{m+\epsilon,m}(\cO_\fp)& M_{m+\epsilon,n-m}(\cO_\fp)\end{bmatrix} B_{\GL_n,\cO_\fp},
\end{align*}
and
\begin{align*}
    &\supp\, \phi_{G_\fp,\breve{\ttt}}\\
    = &\,\left\{X\in M_{2m+\epsilon,n}(K_\fp):\ltrans{\bar{X}} J_{2m+\epsilon} X\in \Her_n(\cO_\fp)\right\}   
    \cap \begin{bmatrix} B^-_{\GL_m,\cO_\fp}B_{\GL_m,\cO_\fp}& M_{n-m,m}(\cO_\fp)\\ M_{m+\epsilon,m}(\cO_\fp)& M_{m+\epsilon,n-m}(\cO_\fp)\end{bmatrix}A(\breve{\ttt})^{-1} U_{\GL_n,\cO_\fp}  \\
     = &\,\left\{X\in M_{2m+\epsilon,n}(K_\fp):\ltrans{\bar{X}} J_{2m+\epsilon} X\in \Her_n(\cO_\fp)\right\}    
    \cap U^-_{H_\fp,\cO_\fp}\begin{bmatrix} \bid_m&0 \\ M_{m+\epsilon,m}(\cO_\fp)& M_{m+\epsilon,n-m}(\cO_\fp)\end{bmatrix}A(\breve{\ttt})^{-1} B_{\GL_n,\cO_\fp}.\\[-1ex]  
\end{align*}

We first check that $\supp\, \phi_{H_\fp,\ttt}\subset \supp\, \phi_{G_\fp,\breve{\ttt}}$, which is easy. Given $X=u^-\ttt^{-1} Y b\in \supp\, \phi_{H_\fp,\ttt}$ with $u^-\in U^-_{H_\fp,\cO_\fp},b\in U_{\GL_n,\cO_\fp},Y\in  \begin{bsm} \bid_m& 0\\ M_{m+\epsilon,m}(\cO_\fp)& M_{m+\epsilon,n-m}(\cO_\fp)\end{bsm}$, 
\[
    \ltrans{\bar{X}}J_{2m+\epsilon}X=\ltrans{\bar{b}}\ltrans{\bar{Y}} J_{2m+\epsilon}Y b \in \Her_n(\cO_\fp),
\]
and $\ttt\in T^+_{H_\fp},\breve{\ttt}\in T^+_{G_\fp}$, plus the condition on $u,d$ in Case (U0) when $n>m$, $K_\fp=F_\fp\oplus F_\fp$, implies that
\begin{align*}
    \ttt^{-1}Y\in  \ttt^{-1} \begin{bmatrix} \bid_m& 0\\ M_{m+\epsilon,m}(\cO_\fp)& M_{m+\epsilon,n-m}(\cO_\fp)\end{bmatrix}
     \subset \begin{bmatrix} \bid_m&0 \\ M_{m+\epsilon,m}(\cO_\fp)& M_{m+\epsilon,n-m}(\cO_\fp)\end{bmatrix}A(\breve{\ttt})^{-1}.
\end{align*}
It follows that $X=u^-\ttt^{-1}Yb\in U^-_{H_\fp,\cO_\fp}\begin{bsm} \bid_m&0 \\ M_{m+\epsilon,m}(\cO_\fp)& M_{m+\epsilon,n-m}(\cO_\fp)\end{bsm}A(\breve{\ttt})^{-1} B_{\GL_n,\cO_\fp}$. Thus, $X\in \supp\,\phi_{G_\fp,\breve{\ttt}}$.\\

It remains to checking that $\supp\, \phi_{G_\fp,\breve{\ttt}}\subset \supp\, \phi_{H_\fp,\ttt}$, for which it suffices to show that every \begin{equation}\label{eq:Xbelong}
   X\in \left\{\begin{array}{ll}\begin{bmatrix} \bid_m&0 \\ M_{m,m}(\cO_\fp)& M_{m,n-m}(\cO_\fp)\end{bmatrix}A(\breve{\ttt})^{-1} 
    &\text{Case (OS0)(U0)}\\[3ex]
  \begin{bmatrix} \bid_m&0 \\ M_{m,m}(\cO_\fp)& M_{m,n-m}(\cO_\fp)\\ 0&M_{1,n-m}(\cO_\fp)\end{bmatrix}A(\breve{\ttt})^{-1}   &\text{Case (OS1)(U1)}
  \end{array}\right.
  \text{  with  } 
   X^* J_{2m+\epsilon} X  \in \Her_n(\cO_\fp),
\end{equation} 
belongs to $\supp\,\phi_{H_\fp,\ttt}$. (After the left translation by certain element in $U^-_{H_\fp,\cO_\fp}$ and the right translation by certain element in  $U_{\GL_n,\cO_\fp} \subset A(\breve{\ttt})^{-1}U_{\GL_n,\cO_\fp} A(\breve{\ttt})$, an element in $\begin{bsm} \bid_m&0 \\ M_{m+\epsilon,m}(\cO_\fp)& M_{m+\epsilon,n-m}(\cO_\fp)\end{bsm}$ becomes the above shape and the property $X^* J_{2m+\epsilon} X  \in \Her_n(\cO_\fp)$ stays invariant under these translations.) Given such an $X$, we can write it as
\begin{align*}
   X&=\left\{\begin{array}{ll}
   \begin{bmatrix}\bid_m&0\\ Y_3&Y_4\end{bmatrix}\begin{bmatrix}a_{\ttt}\\ &d\cdot \bid_{n-m}\end{bmatrix}^{-1},
    &\text{Case (OS0)(U0)},\\[3ex]
   \begin{bmatrix}\bid_m&0\\ Y_3&Y_4\\0&Y_6\end{bmatrix}\begin{bmatrix}a_{\ttt}\\ &d\cdot \bid_{n-m}\end{bmatrix}^{-1},
    &\text{Case (OS1)(U1)},
  \end{array}\right.
   &&\begin{aligned}
   Y_3&\in M_{m,m}(\cO_\fp), \\
   Y_4&\in M_{m,n-m}(\cO_\fp),\\
   Y_6&\in M_{1,n-m}(\cO_\fp),
   \end{aligned}
\end{align*}
with $a_{\ttt}=\diag(a_1,\cdots,a_m)$. Then
\begin{align}
    \label{eq:XJX0} &\ltrans{\bar{X}} \begin{bmatrix} \,&\bid_m\\ \bid_m\end{bmatrix} X
    = \begin{bmatrix}\bar{a}^{-1}_{\ttt}\\ &\bar{d}^{-1}\cdot \bid_{n-m}\end{bmatrix} \begin{bmatrix}Y_3+\ltrans{\bar{Y}}_3&Y_4\\ \ltrans{\bar{Y}}_4&0\end{bmatrix}\begin{bmatrix}a^{-1}_{\ttt}\\ &d^{-1}\cdot \bid_{n-m}\end{bmatrix},
    &&\textup{Case (OS0)(U0)},\\
    \label{eq:XJX1} &\ltrans{\bar{X}} \begin{bmatrix} \,&\bid_m\\ \bid_m\\&&1\end{bmatrix} X
    = \begin{bmatrix}\bar{a}^{-1}_{\ttt}\\ &\bar{d}^{-1}\cdot \bid_{n-m}\end{bmatrix} \begin{bmatrix}Y_3+\ltrans{\bar{Y}}_3&Y_4\\ \ltrans{\bar{Y}}_4&\ltrans{\bar{Y}}_6Y_6\end{bmatrix}\begin{bmatrix}a^{-1}_{\ttt}\\ &d^{-1}\cdot \bid_{n-m}\end{bmatrix}, 
    &&\textup{Case (OS1)(U1)}.
\end{align}

In Case (OS0)(U0), we can write $X$ as
\begin{align}
    \label{eq:X0} X&=\begin{bmatrix}\bid_m&0\\ \frac{Y_3-Y^*_3}{2}&\bid_m\end{bmatrix} 
    \cdot \ttt ^{-1} \begin{bmatrix}a_{\ttt}\\& \bar{a}^{-1}_{\ttt} \end{bmatrix} \begin{bmatrix} \bid_m&0\\ \frac{Y_3+Y^*_3}{2}&Y_4\end{bmatrix} \begin{bmatrix}a^{-1}_{\ttt}\\ &d^{-1}\cdot \bid_{n-m}\end{bmatrix}.
\end{align}
From $Y_3\in M_{m,m}(\cO_\fp)$ and $\eqref{eq:XJX0}\in \Her_n(\cO_\fp)$, we see that 
\begin{align*}
    \eqref{eq:X0}\in U^-_{H_\fp,\cO_\fp}\,\ttt^{-1} \begin{bmatrix} \bid_m& 0\\ M_{m,m}(\cO_\fp)& M_{m,n-m}(\cO_\fp)\end{bmatrix}.
\end{align*}
This proves, in Case (OS0)(U0), that every $X$ as in \eqref{eq:Xbelong} belongs to $\supp\,\phi_{H_\fp,\ttt}$. 

In Case (OS1)(U1), we can write $X$ as
\begin{align*}
    X&=\begin{bmatrix}\bid_m&0\\ \frac{Y_3-Y^*_3}{2}&\bid_m\\&&1\end{bmatrix} 
    \cdot \ttt ^{-1} \begin{bmatrix}a_{\ttt}\\& \bar{a}^{-1}_{\ttt} \\&&u\end{bmatrix} \begin{bmatrix} \bid_m&0\\ \frac{Y_3+Y^*_3}{2}&Y_4\\0&Y_6\end{bmatrix} \begin{bmatrix}a^{-1}_{\ttt}\\ &d^{-1}\cdot \bid_{n-m}\end{bmatrix}.
\end{align*}
From $Y_3\in M_{m,m}(\cO_\fp),Y_6\in M_{1,n-m}(\cO_\fp)$ and \eqref{eq:XJX0} plus the condition on $u,d$ for Case (U1) with $n>m$, $K_\fp=F_\fp\oplus F_\fp$, we see that 
\begin{align*}
    X\in U^-_{H_\fp,\cO_\fp}\,\ttt^{-1} \begin{bmatrix} \bid_m& 0\\ M_{m,m}(\cO_\fp)& M_{m,n-m}(\cO_\fp)\\0&M_{1,n-m}(\cO_\fp)\end{bmatrix}.
\end{align*}
This proves, in Case (OS1)(U1), that every $X$ as in \eqref{eq:Xbelong} belongs to $\supp\,\phi_{H_\fp,\ttt}$.  
\end{proof}

\section{Theta lifts of Hida families on definite orthogonal and unitary groups}

As an application of the properties proved for the test Schwartz functions at $\fp$ in the previous section. We use them to construct theta lifts of Hida families on definite orthogonal and unitary groups to symplectic and quasi-split unitary groups.

\subsection{Some setting}

Let $F$ be a totally real field and let $K=F$ or a quadratic CM extension of $F$. Fix an odd prime number $p$ unramified in $K$. 

Define the algebraic group $G'$ as the subgroup of ${\rm Res}_{K/\bQ}\GL_{n,K}$ such that for a $\bQ$-algebra $R$, 
\[
    G'(R)=\left\{g\in \GL_{2n}(R\otimes_\bQ K):\ltrans{\bar{g}}\begin{bmatrix}&\bid_n\\-\bid_n\end{bmatrix}=\nu(g)\begin{bmatrix}&\bid_n\\-\bid_n\end{bmatrix},\,\nu(g)\in R^\times\right\},
\]
and 
\[
    G=\left\{g\in G':\nu(g)=1\right\}.
\]

Let $W$ be a symmetric space over $K$ when $K=F$ (resp. Hermitian space over $K$ when $K$ is a CM extension of $F$) of dimension $2m+\epsilon$.  Fix a basis $e_1,\dots,e_{2m+\epsilon}$ of $W$ and put $\zeta_W=\left(\left<e_i,e_j\right>\right)_{1\leq i,j\leq 2m+\epsilon}$ (where we use the convention that $\left<aw_1,bw_2\right>=\bar{a}b\left<w_1,w_2\right>$, $a,b\in K$). Define the algebraic group $H$ over $\bQ$ as  the subgroup of ${\rm Res}_{K/\bQ}\GL_{n,K}$ such that for a $\bQ$-algebra $R$,
\[
   H(R)=\left\{g\in \GL_{2m+\epsilon}(R\otimes_{\bQ} K): \ltrans{\bar{g}}\zeta_W g=\zeta_W\right\}.
\]
We assume that $H$ is compact at all archimedean places and quasi-split at at all places above $p$. We fix a Haar measure on $H(\bA_{\bQ,f})$.

Fix an additive character $\psi=\bigotimes_v \psi_v:F\backslash\bA_F\ra\bC^\times$ with 
\[
    \psi_v(x)=\left\{\begin{array}{ll}e^{-2\pi i\{\Tr_{F_v/\bQ_{\ell(v)}}x\}}, &\text{$v$ finite},\\[1ex]
    e^{2\pi i x},& v\mid\infty,\end{array}\right.
\]
where $\ell(v)$ denotes the place of $\bQ$ under $v$ and $\{\Tr_{F_v/\bQ_{\ell(v)}}x\}$ denote the fractional part of $\Tr_{F_v/\bQ_{\ell(v)}}x$.

\subsection{$p$-adic families of automorphic forms on $H$ and Petersson inner product}\label{sec:Hpadic} 
Denote by $\Sigma_p$ the set of places of $F$ diving $p$. We fix an identification between $(W\otimes\bQ_p,\left<,\right>)$ between $\prod_{p\in\Sigma_\fp} (K\otimes_F F_\fp)^{2m+\epsilon}$ together with the symmetric/Hermitian form given by $J_{2m+\epsilon}$ (defined in \eqref{eq:J}) with respect to the standard basis of  $(K\otimes_F F_\fp)^{2m+\epsilon}$. Such an identification also gives
\begin{align*}
    H(\bQ_p)&\xrightarrow[\cong]{\iota_{H,p}}
    \prod_{\fp\in\Sigma_p}H_\fp=\left\{\begin{array}{ll}\prod_{\fp\in\Sigma_p} \OO_{2m+\epsilon,\fp}, & K=F,\\[1ex]
   \prod_{\fp\in\Sigma_p} \U_{2m+\epsilon,\fp}, & K\neq F.\end{array}\right.
\end{align*}

Let $K^{H,p}_f\subset H(\bA^p_{\bQ,f})$ be an open compact subgroup. For a positive integer $r$, let $K^H_{p,1}(p^r)\subset H(\bQ_p)$ (resp. $K^H_{p,1}(p^r)\subset H(\bQ_p)'$) be the open compact subgroup consisting of elements whose images under $\iota_{H,p}$ belong to
\begin{align*}
   &\Bigg\{(h_\fp)\in \prod_{\fp\in\Sigma_p} H_\fp\cap \GL_{2m+\epsilon}(\cO_\fp):h_\fp\equiv \begin{blockarray}{ccc}m&m+\epsilon\\ \begin{block}{[cc]c}A&B&m\\0&D&m+\epsilon\\ \end{block}\end{blockarray}\text{ with }A\in U_{\GL_m,\cO_\fp}\text{ for all $\fp\in\Sigma_p$}\Bigg\}\\
   &(\text{resp. }
   \Bigg\{(h_\fp)\in \prod_{\fp\in\Sigma_p} H_\fp\cap \GL_{2m+\epsilon}(\cO_\fp):h_\fp\equiv \begin{blockarray}{ccc}m&m+\epsilon\\ \begin{block}{[cc]c}A&B&m\\0&D&m+\epsilon\\ \end{block}\end{blockarray}\text{ with }D\in U^-_{\GL_m,\cO_\fp}\text{ for all $\fp\in\Sigma_p$}\Bigg\}). 
\end{align*}
For a $\bZ_p$-algebra $R$, we let
\begin{align*}
   T_{H,p}(R)&=\Big\{t\in H(\bQ_p\otimes_{\bZ_p} R):\iota_{H,p}(t)\text{ is diagonal with entries in $\prod_{\fp\in\Sigma_p} \left(\cO_\fp\otimes_{\bZ_p} R\right)^\times$}\Big\},
\end{align*}
and
\begin{align*}
    K^H_{p,0}(p^r)&=T_{H,p}(\bZ_p)K^H_{p,1}(p^r).
\end{align*}

\subsubsection{$p$-adic families on $H$}

For a ring $R$, define
\begin{equation}\label{eq:MHR}
\begin{aligned}
    M\big(K^{H,p}_fK^H_{p,1}(p^r),R\big)=\left\{\text{functions on $H(\bQ)\backslash H(\bA_{\bQ,f})/K^{H,p}_fK^H_{p,1}(p^r)$ with values in $R$}\right\},
\end{aligned}
\end{equation}
and define the following spaces of $p$-adic automorphic forms on $H$ over $\cO_L$ with tame level $K^{H,p}_f$ 
\begin{align}
   \label{eq:VHpadic} V^H\big(K^{H,p}_f,\cO_L\big)&=\varprojlim_l \varinjlim_r M\big(K^{H,p}_fK^H_{p,1}(p^r),\cO_L/p^l\cO_L\big),\\
\label{eq:VHpdiv}  \pzV^H\big(K^{H,p}_f,\cO_L\big)&=\varinjlim_l \varinjlim_r M\big(K^{H,p}_fK^H_{p,1}(p^r),\cO_L/p^l\cO_L\big).
\end{align}
The torus $T_{H,p}(\bZ_p)$ acts on the above spaces by right translation. 

Given $\ttt_p\in T^+_{H,p}(\bQ_p)=\iota^{-1}_{H,p}\left(\prod_{\fp\in\Sigma_p}T^+_{H_\fp}\right)$ (see \S\ref{sec:Up-oper} for the definition of $T^+_{H_\fp}$), we define the $\bU_p$-operator  $\tilde{U}^{H}_{p,\ttt_p}$ as a renormalization of the  operator defined in \eqref{eq:UHp}:
\begin{equation}\label{eq:tildeUp}
\begin{aligned}
   \tilde{U}^{H_\fp}_{\fp,\ttt_\fp}
   &=\vol\left(\ttt_\fp U_{H_\fp,\cO_\fp}\ttt^{-1}_\fp \right)^{-1}U^{H_\fp}_{\fp,\ttt_\fp}
   = \sum_{u\in U_{H_\fp,\cO_\fp}/\ttt_\fp U_{H_\fp,\cO_\fp}\ttt^{-1}_\fp} \text{action by $u\ttt_\fp$}\\
   \tilde{U}^H_{p,\ttt_p}
   &=\iota^{-1}_{H,p}\Big(\prod_{\fp\in\Sigma_p} \tilde{U}^{H_\fp}_{\fp,\ttt_\fp}\Big).
\end{aligned}
\end{equation}
Thanks to the normalization in \eqref{eq:tildeUp}, they naturally act on $M\big(K^{H,p}_fK^H_{p,1}(p^r),R\big)$ via right translation for all $\bZ_p$-algebra $R$. Fix $\ttt^\star_p=\iota^{-1}_{H,p}\left((\ttt^\star_\fp)_{\fp\in\Sigma_p}\right)$ with
\[
   \ttt^\star_\fp
   =\left\{\begin{array}{ll}
   \diag(p^m,p^{m-1},\cdots,p,p^{-m},p^{-m+1},\cdots,p^{-1}), &\epsilon=0,\\
   \diag(p^m,p^{m-1},\cdots,p,p^{-m},p^{-m+1},\cdots,p^{-1},1), &\epsilon=1,\,\text{with $\fp$ inert in $K$},\\
   \diag(p^{m+1},p^{m},\cdots,p^2,p^{-m-1},p^{-m},\cdots,p^{-2},(p,p^{-1})), &\text{$\epsilon=1$,  with $\fp$ split in $K$}.
   \end{array}\right.
\]
We define the ordinary projector on the spaces defined in \eqref{eq:VHpadic}\eqref{eq:VHpdiv} as
\[
    e^H_\ord=\lim_{s\to\infty} \left(\tilde{U}^H_{p,\ttt^\star_p}\right)^{s!}.
\]
(The convergence follows from that the limit  exists on the finite $\cO_L$-modules \eqref{eq:MHR} with $R=\cO_L/p^l\cO_L$.) By its definition, $e^H_\ord$ projects \eqref{eq:VHpadic}\eqref{eq:VHpdiv} into the subspaces where all $\bU_p$-operators act invertibly.

Take the Pontryagin dual of the ordinary part of \eqref{eq:VHpdiv}:
\begin{equation*}
   \pzV^H_\ord\big(K^{H,p}_f,\cO_L\big)^*=\Hom_{\bZ_p}\left(e^H_\ord\pzV^H\big(K^{H,p}_f,\cO_L\big),\bQ_p/\bZ_p\right),
\end{equation*}
on which the $\cO_L\llb T_{H,p}(\bZ_p)\rrb$-structure on \eqref{eq:VHpdiv} naturally induces a $\cO_L\llb T_{H,p}(\bZ_p)\rrb$-structure.

\begin{prop}
Let $T_{H,p}(\bZ_p)^\circ$ be the connected component of $T_{H,p}(\bZ_p)$ containing $\bid_{2m+\epsilon}$ and $\Lambda_H=\cO_L\llb T_{H,p}(\bZ_p)^\circ\rrb$. The spaces $\pzV^H_\ord\big(K^{H,p}_f,\cO_L\big)^*$ is a free $\Lambda_H$-module of rank equal to $\dim_{k_L}e^H_\ord M\big(K^{H,p}_fK^H_{p,1}(p),k_L\big)$ where $k_L$ denotes the residue field of $L$. 
\end{prop}

\begin{proof}
Our setting is a bit more general than the one in \cite{Geraghty}, but the same argument in the proof of Proposition~2.20 in {\it loc.cit}  applies.
\end{proof}

Define the space of ($\Lambda_H$-adic) Hida families on $H$ of tame level $K^{H,p}_f$ as
\[
    \pzM^H_\ord(K^{H,p}_f,\cO_L)=\Hom_{\Lambda_H}\left(\pzV^H_\ord(K^{H,p}_f,\cO_L)^*,\Lambda_H\right).
\]

For our later construction, we define another $\cO_L\llb T_{H,p}(\bZ_p)\rrb$-module. Denote by
\[
   \Meas\left(T_{H,p}(\bZ_p),V^H\big(K^{H,p}_f,\cO_L\big)\right)
\] 
the space of $p$-adic measures on $T_{H,p}(\bZ_p)$ valued in $V^H\big(K^{H,p}_f,\cO_L\big)$ ({\it i.e.} continuous $\cO_L$-linear maps from $\sC\left(T_{H,p}(\bZ_p),\cO_L\right)$, the space of continuous $\cO_L$-valued functions on $T_{H,p}(\bZ_p)$ equipped with the topology of uniform convergence, to $V^H\big(K^{H,p}_f,\cO_L\big)$ equipped with $p$-adic topology). Inside it, we consider the subspace 
\begin{equation}\label{eq:MeasH}
\begin{aligned}
    &\Meas\left(T_{H,p}(\bZ_p),V^H\big(K^{H,p}_f,\cO_L\big)\right)^\natural\\
    =&\,\left\{\mu\in \Meas\left(T_{H,p}(\bZ_p),V^H\big(K^{H,p}_f,\cO_L\big)\right):
   \,\begin{aligned}&\mu(t\cdot \phi)=t\cdot\mu(\phi)\text{ for all  $t\in T_{H,p}(\bZ_p)$}\\ 
   &\phi\in \sC\left(T_{H,p}(\bZ_p),\cO_L\right)\end{aligned} \right\}
\end{aligned}
\end{equation}
where the action of $T_{H,p}(\bZ_p)$ on  $\sC\left(T_{H,p}(\bZ_p),\cO_L\right)$ is via translation. This space comes with a natural $\cO_L\llb T_{H,p}(\bZ_p)\rrb$-structure compatible with the $H(\bA_{\bQ,f})$-action given by
\begin{align*}
   (h\cdot\mu)(\phi)&=h\cdot \mu(\phi),
   &h\in H(\bA_{\bQ,f}).
\end{align*}

\subsubsection{A $p$-adic Petersson inner product}\label{sec:Pet}

Let $\vartheta_{H,p}=\iota^{-1}_{H,p}\left((\vartheta_{H_\fp})_{\fp\in\Sigma_p}\right)$ with 
\[
    \vartheta_{H_\fp}=\left\{\begin{array}{ll}
    \begin{bmatrix}\phantom{0}&\bid_m\\ \bid_m\end{bmatrix}, &\epsilon=0,\\[2ex]
    \begin{bmatrix}\phantom{0}&\bid_m\\ \bid_m\\&& 1\end{bmatrix}, &\epsilon=1.
    \end{array}\right.
\]
We denote by $\Hom_{\Lambda_H}\left(\pzV^H(K^{H,p}_f,\cO_L)^*,\Lambda_H\right)'$ the $\cO_L\llb T_{H,p}(\bZ_p)\rrb$-module which, as a set with an action of $H(\bA_{\bQ,f})$, is the same as $\Hom_{\Lambda_H}\left(\pzV^H(K^{H,p}_f,\cO_L)^*,\Lambda_H\right)$ and with the $\cO_L\llb T_{H,p}(\bZ_p)\rrb$-module structure given by $t\in T_{H,p}(\bZ_p)$ acting through $\bar{t}$. 
In this subsection, we define a pairing between the spaces $\pzM^H_\ord(K^{H,p}_f,\cO_L)'=e^H_\ord\Hom_{\Lambda_H}\left(\pzV^H(K^{H,p}_f,\cO_L)^*,\Lambda_H\right)'$ and $\Meas\left(T_{H,p}(\bZ_p),V^H\big(K^{H,p}_f,\cO_L\big)\right)^\natural$.

For $h\in H(\bA_{\bQ,f})$, we have
\begin{align*}
    &{\rm ev}_h\in \pzV^H(K^{H,p}_f,\cO_L)^*,
    &&{\rm ev}_h:V^H\big(K^{H,p}_f,\cO_L\big)\lra \cO_L
\end{align*}
given by evaluation at $h$. For 
\begin{align*}
   \mu&\in \Meas\left(T_{H,p}(\bZ_p),V^H\big(K^{H,p}_f,\cO_L\big)\right)^\natural,
   &\eta&\in\Hom_{\Lambda_H}\left(\pzV^H(K^{H,p}_f,\cO_L)^*,\Lambda_H\right)',
\end{align*}
we let $\mu_h\in \Meas\left(T_{H,p}(\bZ_p),\cO_L\right)$ be the image of $\mu$ under the map $\Meas\left(T_{H,p}(\bZ_p),V^H\big(K^{H,p}_f,\cO_L\big)\right)\ra \Meas\left(T_{H,p}(\bZ_p),\cO_L\right)$ induced by ${\rm ev}_h$, and $\eta_h=\eta({\rm ev}_h)\in\Lambda_H$. Then for all $t\in \cO_L\llb T_{H,p}(\bZ_p)\rrb$ 
\begin{align}
  \label{eq:Tinv}  \mu_{ht}\otimes\eta_{ht\vartheta_{H,p}\ttt_p}
    &= t\cdot \mu_h\otimes \eta_{ht\vartheta_{H,p}\ttt_p}
    = \mu_h\otimes \bar{t}\cdot   \eta_{ht\vartheta_{H,p}\ttt_p}\\
   \nonumber &= \mu_h\otimes \bar{t}\cdot   \eta_{h\vartheta_{H,p}\bar{t}^{-1}\ttt_p}= \mu_h\otimes   \eta_{h\vartheta_{H,p}\bar{t}^{-1}\ttt_p\bar{t}}
    =\mu_h\otimes\eta_{h\vartheta_{H,p}\ttt_p},
\end{align}
as  elements in $ \Meas\left(T_{H,p}(\bZ_p),\cO_L\right)\otimes_{\cO_L\llb T_{H,p}(\bZ_p)\rrb} \Lambda_H\cong {\cO_L\llb T_{H,p}(\bZ_p)\rrb}$. 

\vspace{.5em}

Given $r\in\bZ_{\geq 1}$ and  $\ttt_p=\iota^{-1}_{H,p}\left((\ttt_\fp)_{\fp\in\Sigma_p}\right)\in T^+_{H,p}$ sufficiently positive for $r$ such that 
\[
    K^H_{p,0}(p)\,\cap\, (\vartheta_{H,p}\ttt_p) K^H_{p,0}(p)(\vartheta_{H,p}\ttt_p)^{-1}=K^H_{p,0}(p)\,\cap\, (\vartheta_{H,p}\ttt_p) K^H_{p,0}(p^r)(\vartheta_{H,p}\ttt_p)^{-1}\subset K^H_{p,0}(p^r),
\]
Thanks to \eqref{eq:Tinv}, we can define
\begin{align*}
    \sB_{K^{H,p}_f,\,\ttt_p}: \Meas\left(T_{H,p}(\bZ_p),\cO_L\right)\otimes_{\cO_L\llb T_{H,p}(\bZ_p)\rrb} \Hom_{\Lambda_H}\left(\pzV^H(K^{H,p}_f,\cO_L)^*,\Lambda_H\right)'
    \lra  \cO_L\llb T_{H,p}(\bZ_p)\rrb
\end{align*}
as
\[
    \sB_{K^{H,p}_f,\,\ttt_p}(\mu\otimes\eta)
    = \sum_{h\in H(\bQ) \backslash H(\bA_{\bQ,f})/K^{H,p}_f \left(K^H_{p,0}(p)\,\cap\, (\vartheta_{H,p}\ttt_p) K^H_{p,0}(p)(\vartheta_{H,p}\ttt_p)^{-1}\right)}\mu_h\otimes\eta_{h\vartheta_{H,p}  \ttt_p}.
\]
It is straightforward to check that for $\ttt_p,\ttt'_p\in T^+_{H,p}$ both sufficiently positive for $r$, we have 
\[
    \sB_{K^{H,p}_f,\,\ttt_p\ttt'_p}(\mu\otimes\eta)
    =\sB_{K^{H,p}_f,\,\ttt_p}(\mu\otimes\tilde{U}_{p,\ttt'_p}\eta).
\]  
Therefore, if $\eta$ further belongs to $\pzM^H_\ord(K^{H,p}_f,\cO_L)'=e^H_\ord\Hom_{\Lambda_H}\left(\pzV^H(K^{H,p}_f,\cO_L)^*,\Lambda_H\right)'$, we can define 
\begin{align}
     &\sB_{K^{H,p}_f}(\mu\otimes\eta)
    =\sB_{K^{H,p}_f,\,(\ttt^\star_\fp)^s}(\mu\otimes(\tilde{U}_{p,\ttt^\star_p})^{-s}\eta).
    && s\gg 0.
\end{align}
This gives us an $\cO_L\llb T_{H,p}(\bZ_p)\rrb$-linear functional
\begin{equation}\label{eq:sB} 
    \sB_{K^{H,p}_f}:\Meas\left(T_{H,p}(\bZ_p),\cO_L\right)\otimes_{\cO_L\llb T_{H,p}(\bZ_p)\rrb} \pzM^H_\ord(K^{H,p}_f,\cO_L)'\lra \cO_L\llb T_{H,p}(\bZ_p)\rrb,
\end{equation}
whose specialization to classical automorphic forms is the Petersson inner product with a modification at places above $p$.

\subsection{Constructing $p$-adic families of $q$-expansions on $G$}

\subsubsection{Interpolating Fourier coefficients of theta series}

Let $\Her_n(\bA_K)=\{\sigma\in M_{n,n}(\bA_K):\sigma=\ltrans{\bar{\sigma}}\}$, the set of symmetric $n\times n$ Hermitain (resp. symmetric) matrices when $K\neq F$ (resp. $K=F$). We fix the Haar measure on it with 
\begin{align*}
  & \vol(\Her_n(\cO_{K,v}))=1, \quad \quad\text{$v$ finite place of $F$},\\
  &2^{n(n-1)[F:\bQ]/2}|D_F|^{-n/2}|D_K|^{-n(n-1)/4}\cdot \prod_{v\mid \infty}\left(\prod_{j=1}^n dx_{v,jj}\prod_{1\leq i< j\leq n} 2^{-1}dz_{v,ij}d\bar{z}_{v,ij}\right), \quad \quad\text{at $\infty$}.
\end{align*}

Let $\phi=\bigotimes_v \phi_v$ be a Schwartz function on $W\otimes_K\bA_K$ with $\phi_v=\mathds{1}_{W^n_0\otimes_{\cO_F} \cO_{F,v}}$ for almost all finite places $v$ of $F$, where $W_0=\cO_K e_1+\cO_K e_2+\cdots +\cO_K e_{2m+\epsilon}$ (with $e_1,e_2,\dots,e_{2m+\epsilon}$ our fixed basis of $W$). Consider the theta series on $G(K)\times H(K)\backslash G(\bA_K)\times H(\bA_K)$ associated to $\phi$:
\begin{equation}\label{eq:thetaphi}
    \theta(\phi)(g,h)=\sum_{X\in W} \omega_\psi(g,h)\theta(\phi),
\end{equation}
where $\omega_\psi$ denotes the Weil representation (see \S\ref{sec:WeilRep} for some formulas). Take $\beta\in\Her_n(K)$, the $\beta$-th Fourier coefficient of $\theta(\phi)(\cdot,h)$ at $g\in G(\bA_K)$ is computed by
\begin{equation}\label{eq:FC-intg}
      \theta(\phi)_\beta(g,h)=\int_{\Her_n(K)\backslash \Her_n(\bA_K)}\theta(\phi)\left(\begin{bmatrix}\bid_n&y\\ 0&\bid_n\end{bmatrix}g,h\right) \cdot \psi(-\Tr\beta y)\,dy.
\end{equation}
Plugging \eqref{eq:thetaphi} into \eqref{eq:FC-intg},
\begin{equation}\label{eq:theta-beta}
\begin{aligned}
     \theta(\phi)_\beta(g,h)
     &=\int_{\Her_n(K)\backslash \Her_n(\bA_K)} 
   \sum_{X\in W^n}  \left( \omega_\psi \left(\begin{bmatrix}\bid_n&y\\ 0&\bid_n\end{bmatrix},\bid_H\right)
     \left(\omega_\psi\left(g,h\right)\phi\right)\right)(X)   \cdot \psi(-\Tr\beta y)\,dy \\
     &=\int_{\Her_n(K)\backslash \Her_n(\bA_K)} 
   \sum_{X\in W^n}  
     \left(\omega_\psi\left(g,h\right)\phi\right)(X)   \cdot \psi(\Tr \left<X,X\right>y-\Tr\beta y)\,dy \\
     &=\sum_{X\in W^n,\,\left<X,X\right>_W=\beta}  \left(\omega_\psi\left(g,h\right)\phi\right)(X) \\
     &=\sum_{X\in W^n,\,\left<X,X\right>_W=\beta}  \left(\omega_\psi\left(g\right)\phi\right)(\ltrans{h}X).
\end{aligned}
\end{equation}

For an archimedean place $\sigma:F\ra\bR$, fix an extension $\sigma:K\ra\bR$,  put
\begin{align*}
    g_{z_\sigma}&=\begin{bmatrix}\left(\frac{z_\sigma-\ltrans{\bar{z}_\sigma}}{2i}\right)^{1/2}& \frac{z_\sigma+\ltrans{\bar{z}_\sigma}}{2}\left(\frac{z_\sigma-\ltrans{\bar{z}_\sigma}}{2i}\right)^{-1/2}\\0&\left(\frac{z_\sigma-\ltrans{\bar{z}_\sigma}}{2i}\right)^{-1/2} \end{bmatrix},
     &z_\sigma=x_\sigma+iy_\sigma,\,x_\sigma,y_\sigma\in M_{n,n}(\bR),\,i(\ltrans{\bar{z}_\sigma}-z_\sigma)>0,
\end{align*} 
and define the Schwartz function $\phi^\sG_\sigma$ on $W^n\otimes_{F,\sigma}\bR$ as the Gaussian function:
\begin{align}
   \label{eq:Gaussian} \phi^{\sG}_\sigma(X\otimes a)
    &=e^{-\pi\left<a\sigma(X),a\sigma(X)\right>_W},
   &&X\in W^n.
\end{align}

\begin{prop}\label{prop:mu-beta}
Fix a Schwartz function $\phi^p_f$ on $W^n\otimes_F\bA^p_{\bF,f}$ taking values inside $\cO_L\cap \ol{\bQ}$ and fixed by tame level group $K^{H,p}_f$. For each element $g^p_f= \begin{bmatrix}A^p_f&B^p_f\\ 0&\ltrans{\bar{A}}^{p,-1}_f\end{bmatrix} \in G(\bA^p_{\bQ,f})$, there exists 
\[
    \mu_{g^p_f,\beta}\in \Meas\left(T_{H,p}(\bZ_p),V^H(K^{H,p}_f,\cO_L)\right)^\natural
\]
such that for $h\in H(\bA_{\bQ,f})$ and finite-order character $\uchi:T_{H,p}(\bZ_p)\ra\cO^\times_L$,
\begin{equation}\label{eq:mu-beta}
\begin{aligned}
    \mu_{g^p_f,\beta}(\uchi)=&\,\theta\left(\phi^p_f\otimes \prod\nolimits_{\fp\mid p}\phi_{\uchi_\fp}\otimes\prod\nolimits_{\sigma:F\ra\bR} \phi^\sG_\sigma\right)_\beta(g^p_f \prod_{\sigma:F\ra \bR}g_{z_\sigma},-) \\
    &\times \det\left(\frac{z_\sigma-\ltrans{\bar{z}}_\sigma}{2i}\right)^{-\frac{2m+\epsilon}{2}} e^{-\pi i  \Tr z_\sigma\sigma(\beta)},
\end{aligned}
\end{equation}
where $\phi_{\uchi_\fp}$ is defined in \S\ref{sec:Schwartz} and $\phi^\sG_\sigma$ is defined in \eqref{eq:Gaussian}. 
\end{prop}

\begin{proof}
By the definition of $\phi_{\uchi_\fp}$ in \S\ref{sec:Schwartz} and  \eqref{eq:theta-beta}, for each $h_f\in H(\bA_{\bQ,f})$, we have
\begin{align*}
   \text{RHS of \eqref{eq:mu-beta}}(h_f)
   = &\,\sum_{X\in W^n,\left<X,X\right>_W=\beta}\phi^p_f(\ltrans{h}^p_f X)\prod_{\fp\mid p}\phi_{\uchi_\fp}(\ltrans{h}_\fp X)\\
   &\times \prod_{\sigma:F\ra\bR} \omega_\psi(g_{z_\sigma})\phi^\sG_\sigma(X) \det\left(\frac{z_\sigma-\ltrans{\bar{z}}_\sigma}{2i}\right)^{-\frac{2m+\epsilon}{2}} e^{-\pi i  \Tr z_\sigma\sigma(\beta)}.
\end{align*}
By the formula for $\omega_\psi(g_{z_\sigma})$ and the definition of $\phi^\sG_\sigma$ \eqref{eq:Gaussian},for $X\in W^n$ with $\left<X,X\right>_W=\beta$, we have
\begin{align*}
    \omega_\psi(g_{z_\sigma})\phi^\sG_\sigma(X)
    &= \det\left(\frac{z_\sigma-\ltrans{\bar{z}}_\sigma}{2i}\right)^{\frac{2m+\epsilon}{2}}
     e^{2\pi i \cdot \frac{1}{2}\Tr\left(\frac{z_\sigma+\ltrans{\bar{z}}_\sigma}{2}\left<\sigma(X),\sigma(X)\right>_W\right)}
     e^{-\pi\Tr\left<\left(\frac{z_\sigma-\ltrans{\bar{z}_\sigma}}{2i}\right)^{1/2}\sigma(X),\left(\frac{z_\sigma-\ltrans{\bar{z}_\sigma}}{2i}\right)^{1/2}\sigma(X)\right>_W}\\
     &=\det\left(\frac{z_\sigma-\ltrans{\bar{z}}_\sigma}{2i}\right)^{\frac{2m+\epsilon}{2}}
     e^{\pi i \Tr\left(\frac{z_\sigma+\ltrans{\bar{z}}_\sigma}{2} \sigma(\beta)\right)}
     e^{-\pi\Tr \left(\frac{z_\sigma-\ltrans{\bar{z}}_\sigma}{2i} \sigma(\beta)\right)},\\
     &= \det\left(\frac{z_\sigma-\ltrans{\bar{z}}_\sigma}{2i}\right)^{\frac{2m+\epsilon}{2}} e^{\pi i  \Tr z_\sigma\sigma(\beta)}
\end{align*}

Plugging it into the above equation gives
\begin{align*}
   \textup{RHS of \eqref{eq:mu-beta}}(h)
   = &\,\sum_{X\in W^n,\left<X,X\right>_W=\beta}\phi^p_f(\ltrans{h}^p_fX)\prod_{\fp\mid p}\phi_{\uchi_\fp}(\ltrans{h}_\fp X).
\end{align*}
By our choice, $\phi^p_f(\ltrans{h}^p_f X)$ does not depend on $\uchi$, and $\prod_{\fp\mid p}\phi_{\uchi_\fp}(\ltrans{h}_\fp X)$ can be interpolated by a $p$-adic measure. Thus, for each $h_f\in H(\bA_{\bQ,f})$, RHS of $\eqref{eq:mu-beta}(h_f)\in \Meas(T_{H,p}(\bZ_p),\cO_L)$, and one can see that these measures vary continuouly with respect to $h_f$, so there exists $ \mu_{g^p_f,\beta}\in \Meas\left(T_{H,p}(\bZ_p),V^H(K^{H,p}_f,\cO_L)\right)$ with the interpolation properties given in \eqref{eq:mu-beta}. By the definition of $\phi_{\uchi_\fp}$, we have $\phi_{\uchi_\fp}(\ltrans{(h_\fp t_\fp)}X)=\uchi_\fp(t_\fp)\phi_{\uchi_\fp t_\fp}(\ltrans{h}_\fp)$ for $t_\fp\in T_{H,\fp}(\bZ_p)$, which implies that the measure $\mu_{g^p_f,\beta}$ belongs to $\Meas\left(T_{H,p}(\bZ_p),V^H(K^{H,p}_f,\cO_L)\right)^\natural$.
\end{proof}

Before we continue, we lay out some more setting for our construction of theta lifts of Haida families on $H$. 

\subsubsection{Some more setting}\label{sec:globalsetting}
From now on, we fix an open compact subgroup  $K^{H,p}_f\subset H(\bA^p_{\bQ,f})$, and an element
\begin{equation}\label{eq:bvarphi}
   \bvarphi\in \pzM^H_\ord(K^{H,p}_f,\cO_L)\otimes_{\cO_L\llb T_{H,p}(\bZ_p)\rrb} \bI_{\bvarphi},
\end{equation}
where $\bI_{\bvarphi}$ is the integral closure of $\Lambda_H$ inside a finite field extension of one component of the total ring of fractions of  $\cO_L\llb T_{H,p}(\bZ_p)\rrb$. (Usually one works with a $\bvarphi$ that is an eigenvector for Hecke algebras away from places where $K^H_v\neq H(F_v)\cap \GL_{2m+\epsilon}(\cO_{K,v})$, but this is not essentially needed for carrying out the following construction.)

For a point $\lambda:\bI_{\bvarphi}\ra\ol{\bQ}_p$, we call the character
\[
   T_{H,p}(\bZ_p)\ra \cO_L\llb T_{H,p}(\bZ_p)\rrb^\times\ra \bI^\times_{\bvarphi}\stackrel{\lambda}{\ra} \bar{\bQ}^\times_p
\]
its projection to the weight space.   

\subsubsection{Families of $q$-expansions of theta lifts}

Pairing the measure in Proposition~\ref{prop:mu-beta} with a Hida family on $H$ via the $p$-adic Petersson inner product constructed in \S\ref{sec:Pet} gives rise to a family of theta lifts from $H$ to $G$ in terms of $q$-expansions.

\begin{prop}\label{prop:qexp}
Fix a Schwartz function $\phi^p_f$ on $W^n\otimes_F\bA^p_{\bF,f}$ taking values inside $\cO_L\cap \ol{\bQ}$ and fixed by tame level group $K^{H,p}_f$. For a given $\bI_{\bvarphi}$-adic Hida family $\bvarphi$ as in \eqref{eq:bvarphi}, there exists an $\bI_{\bvarphi}$-adic family of $q$-expansions 
\[
    {\bm \theta}(\bvarphi)_{\qexp, g^p_f}\in\bI_{\bvarphi} \llb \Her_n(K)_{\geq 0}\rrb 
\]
such that for all points $\lambda:\bI_{\bvarphi}\ra\ol{\bQ}_p$ with the composition $T_{H,p}(\bZ_p)\ra \cO_L\llb T_{H,p}(\bZ_p)\rrb^\times\ra \bI^\times_{\bvarphi}\stackrel{\lambda}{\ra} \bar{\bQ}^\times_p$ a finite-order character $\uchi$, 
\begin{align*}
   \lambda\left({\bm \theta}(\bvarphi)_{\qexp, g^p_f}\right)
   =\vol(K^{H,p}_f)^{-1}\cdot \textup{$q$-expansion of } \theta^\dagger\left(\phi^p_f\otimes \prod\nolimits_{\fp\mid p}\phi_{\uchi_\fp}\otimes\prod\nolimits_{\sigma:F\ra\bR} \phi^\sG_\sigma,\lambda(\bvarphi)\right) \text{ at $g^p_f$}. 
\end{align*}
Here for a Schwartz function $\phi$ on $W^n\otimes_F\bA_F$ and an automorphic form $\varphi$ on $H$ which is an eigenvector for the $\bU_p$-operators with eigenvalue $\alpha(\varphi,\ttt^\star_p)$ for $\tilde{U}^H_{p,\ttt^\star_p}$, the theta lift $\theta^\dagger(\phi,\varphi)$ with modification at $p$ is defined by 
\begin{align}
    \label{eq:thetadagger} \theta^\dagger(\phi,\varphi)(g)&=\alpha(\varphi,\ttt^\star_p)^{-s}
     \int_{[H]} \theta \left(\phi)(g,h)\,\varphi(h\vartheta_{H,p}(\ttt^\star_p)^s\right)\,dh,
     &&s\gg 0.
\end{align}
\end{prop}

\begin{proof}
For each $\beta\in \Her_n(K)_{\geq 0}$, viewing $\bvarphi$ as element in $\pzM^H_\ord(K^{H,p}_f,\cO_L)'\otimes_{\cO_L\llb T_{H,p}(\bZ_p)\rrb} \bI_{\bvarphi}$ and applying the pairing \eqref{eq:sB} to the measure $\mu_{g^p_f,\beta}$ in Proposition~\ref{prop:mu-beta} and $\bvarphi$, we obtain 
\[
   \sB_{K^{H,p}_f}\left(\mu_{g^p_f,\beta}\otimes\bvarphi\right)\in \bI_{\bvarphi},
\]
which, by the definition of $\sB_{K^{H,p}_f}$ and the interpolation properties of $\mu_{g^p_f,\beta}$, interpolates the $\beta$-th coefficient in the $q$-expansion of $\theta^\dagger\left(\phi^p_f\otimes \prod\nolimits_{\fp\mid p}\phi_{\uchi_\fp}\otimes\prod\nolimits_{\sigma:F\ra\bR} \phi^\sG_\sigma,\lambda(\bvarphi)\right)$ at the cusp indexed by $g^p_f$, multiplied by $\vol(K^{H,p}_f)^{-1}$. Putting these $\sB_{K^{H,p}_f}\left(\mu_{g^p_f,\beta}\otimes\bvarphi\right)$'s together gives the desired family of $q$-expansions $ {\bm \theta}(\bvarphi)_{\qexp, g^p_f}$.
\end{proof}

\subsection{Theta lifts of Hida families from $H$ to $G$}
When suitable Hida theory on $G$ is available, we can produce a Hida family on $G$ from the $\bI_{\bvarphi}$-adic family of $q$-expansion in Proposition~\ref{prop:qexp}.

\subsubsection{$P$-ordinarity}\label{sec:Pord}
Define the following parabolic subgroup $P\subset \GL(n)$:
\begin{align*}
   P=\left\{\begin{blockarray}{ccc}m&n-m\\ \begin{block}{[cc]c}\begin{smallmatrix}\ast&\cdots&\ast\\ &\ddots&\vdots\\ &&\ast \end{smallmatrix} &\begin{smallmatrix}\ast&\cdots&\ast\\ \vdots &\ddots&\vdots\\ \ast&\cdots&\ast \end{smallmatrix}&m\\[3ex] &\begin{smallmatrix}\ast&\cdots&\ast\\ \vdots &\ddots&\vdots\\ \ast&\cdots&\ast \end{smallmatrix}&n-m\\ \end{block}\end{blockarray}\in \GL(n)  \right\},
\end{align*}
and
\begin{align}
   \label{eq:SP} SP=\left\{ \begin{bmatrix}\begin{smallmatrix}1&\cdots&\ast\\ &\ddots&\vdots\\&&1\end{smallmatrix} &\ast\\ & A_2\end{bmatrix}\in P: A_2\in \SL(n-m)\right\}.
\end{align}
Denote by $P_{\cO_\fp},SP_{\cO_\fp}$ (resp.  $SP_{\cO_\fp/p^r}$)  the subgroup of the $\fp$-component of $P(\bQ_p),SP(\bQ_p)$ consisting of elements with entries in $\cO_\fp$ (resp. $\cO_\fp/p^r$) and entries along the diagonal in $\cO^\times_\fp$ (resp. $(\cO_\fp/p^r)^\times$). Write $T_P(\bZ_p)=\prod_{\fp\mid p} P_{\cO_\fp}/ SP_{\cO_\fp}$.

Let $\sT_{l,r}$ be the level $r$ Igusa tower over $\cO_L/p^l\cO_L$ over the ordinary locus of the Shimura variety for $G'$ of tame level $K^{G',p}_f\subset G'(\bA^p_{\bA,f})$. Let
\[
     V^{G',n-m}_{P,l,r}\big(K^{G',p}_f,\cO_L\big)
    =H^0\left(\sT_{l,r}/K_{SP,\fp}(p^r),\cI^{n-m}_{l,r}\right)
\]
with $K_{SP,\fp}(p^r)=\left\{A\in \GL(n,\cO_\fp):A\mod p^r \text{ belongs to } SP(\cO_\fp/p^r)\right\}$, and $\cI^{n-m}_{l,r}$ the pullback to $\sT_{l,r}$ of the sheaf of ideals associated to the union of the strata indexed by a fully symplectic admissible filtration with the isotropic part of rank $> n-m$ on the Shimura variety for $G'$. Put
\begin{align*}
   \pzV^{G',n-m}_P\big(K^{G',p}_f,\cO_L\big)&=\varinjlim_l\varinjlim_r  V^{G',n-m}_{P,l,r},
\end{align*}
which is naturally a module over $\cO_L\llb T_P(\bZ_p) \rrb$.

Consider the following sub-semigroup of $T^+_{G_\fp}$ (defined in \eqref{eq:T+GH}):
\begin{align*}
    T^+_{P,\fp}&=\left\{\breve{\ttt}\in T^+_{G,\fp}: \breve{\ttt} P_{\cO_\fp} \breve{\ttt}^{-1} \subset P_{\cO_\fp}\right\}\\
    &=\left\{\diag(a_1,\dots,a_m, d,\dots,d,\bar{a}^{-1}_1,\dots,\bar{a}^{-1}_m,\bar{d}^{-1},\dots,\bar{d}^{-1})\in T^+_{G,\fp}\right\}.
\end{align*}
For each element $\breve{\ttt}_\fp\in  T^+_{P,\fp}\subset T^+_{G_\fp}$, there is an optimally normalized $\bU_p$-operator $\tilde{U}^G_{\fp,\breve{\ttt}^\star_\fp}$ acting on $V^{G'}_{P,l,r}$, which agrees with the $\bU_p$-operator $U^G_{\fp,\breve{\ttt}^\star_\fp}$ defined in \eqref{eq:UGp} up to a power of $p$ when acting on classical automorphic forms on $G'$. Fix 
\[
    \breve{\ttt}^\star_\fp
   =\diag(p^{m+1},p^m,\cdots,p^2,p,\dots,p,p^{-m-1},p^{-m},\cdots,p^{-2},p^{-1},\dots,p^{-1}),
\]
and define the $P$-ordinary projector
\begin{equation}\label{eq:eG-Pord}
    e^G_{\Pord}=\lim_{s\ra\infty}  \left(\tilde{U}^G_{\fp,\breve{\ttt}^\star_\fp}\right)^s.
\end{equation}

\subsubsection{Hida theory for automorphic forms on $G'$}\label{sec:HidaTheory}

A (non-cuspidal $P$-ordinary) Hida theory aims to establish the following statements:
\begin{enumerate}
\item[--] The limit \eqref{eq:eG-Pord} exists on $\pzV^{G',n-m}_P$.
\item[--] The $\cO_L\llb T_P(\bZ_p)\rrb$-module 
\[
    \pzV^{G',n-m}_{\Pord}\big(K^{G',p}_f,\cO_L\big)^*
     =  e^G_{\Pord}\Hom_{\bZ_p}\left(\pzV^{G',n-m}_P\big(K^{G',p}_f,\cO_L\big),\bQ_p/\bZ_p\right)
\] 
is free of finite rank over $\Lambda_P=\cO_L\llb T_P(\bZ_p)^\circ\rrb$, where $T_P(\bZ_p)^\circ$ is the connected component of $T_P(\bZ_p)$. 
\item[--] Defining the space of $P$-ordinary Hida families of automorphic forms of tame level $K^{G',p}_f$ on $G'$ as
\[ 
     \pzM^{G',n-m}_{\Pord}(K^{G',p}_f,\cO_L)=\Hom_{\Lambda_P}\left(\pzV^{G',n-m}_{\Pord}\big(K^{G',p}_f,\cO_L\big)^*,\Lambda_P\right),
\] 
for all $\tau\in \Hom(T_P(\bZ_p),\ol{\bQ}^\times_p)$ that can be written as a product of an algebraic character and a finite-order character, there is a natural embedding 
\begin{align}
   \label{eq:Ptau-embd} M^{G',n-m}_{\Pord}(K^{G',p}_f,\tau,\cO_{L(\tau)})\lhra \pzM^{G',n-m}_{\Pord}(K^{G',p}_f,\cO_L)\otimes_{\cO_L\llb T_P(\bZ_p)\rrb} \cO_L\llb T_P(\bZ_p)\rrb/\cP_{\tau},
\end{align}
where $\cP_\tau\subset \cO_L\llb T_P(\bZ_p)\rrb$ is the ideal corresponding to $\tau$, $M^{G',n-m}_{\Pord}(K^{G',p}_f,\tau,\cO_{L(\tau)})$ is the $\cO_{L(\tau)}$-module spanned by $P$-ordinary classical holomorphic automorphic forms on $G'$ of tame level $K^{G',p}_f$ and weight $\tau$ with coefficient of $q$-expansions belonging to $\cO_{L(\tau)}=\cO_L \llb T_P(\bZ_p)\rrb/\cP_{\tau}$.

\end{enumerate}
Assuming that $p$ is an odd prime and unramified in $K$, Such a Hida theory has been established in the following cases:
\begin{enumerate}
\item $K=F=\bQ$ \cite{HidaCon,PiHida,LRtz},
\item $m=n$, $K/F$ is a quadratic CM extension, every places of $F$ above $p$ splits in $K$  \cite{HidaCon, Marcil}.
\end{enumerate}

\subsubsection{Theta lifts of Hida families}

Let $SP_{G',\bZ_p}=\left\{\begin{bmatrix}A&B\\0&D\end{bmatrix}\in G'(\bZ_p):A\in \prod_{\fp\mid p}SP_{\cO_\fp}\right\}$ (with $SP_{\cO_\fp}$ defined below \eqref{eq:SP}) and $SP_{G,\bZ_p}=SP_{G',\bZ_p}\cap G(\bZ_p)$. Let $K^G_\infty=\prod_{\sigma:F\ra\bR} K^G_\sigma$ with $K^G_\sigma$ consisting of $\begin{bmatrix}A&B\\ -\ltrans{\bar{B}}&\bar{A}\end{bmatrix}\in \GL_{2n}(K_\sigma)$ with $A+iB$ preserving the symmetric (resp. Hermitian) form given by $\bid_{2n}$, and let $K^{G'}_\infty$ be the product of $K^G_\infty$ and the connected component of $\bid_{2n}$ of the center of $G'(\bR)$. 

We denote by $\imath_{G,G'}$ the extension by zero corresponding to the embedding 
\begin{equation}\label{eq:imathembd}
G(\bQ)\backslash G(\bA_\bQ)/K^{G,p}_f SP_{G,\bZ_p} K^G_\infty\hra G'(\bQ)\backslash G'(\bA_\bQ)/K^{G',p}_f SP_{G',\bZ_p} K^{G'}_\infty.
\end{equation}

\begin{thm}\label{thm:family}
Suppose that we are in one of the cases listed at the end of  \S\ref{sec:HidaTheory}. In the same setting as Proposition~\ref{prop:qexp}, take a tame level group $K^{G',p}_f\subset \prod_{v\neq p \text{ finite place of $\bQ$}}G'(\bZ_v)$ such that the Schwartz function $\phi^p_f$ is fixed by $K^{G,p}_f=K^{G'_p}_f\cap G(\bA^p_{\bQ,f})$. There exists an $\bI_{\bvarphi}$-adic $P$-ordinary family
\[
    {\bm \theta}(\bvarphi)\in  \pzM^{G',n-m}_{\Pord}(K^{G',p}_f,\cO_L)\otimes_{\Lambda_P} \bI_{\bvarphi}
\]
interpolating the theta lifts of finite-order points of $\bvarphi$, {\it i.e.} for all points $\lambda:\bI_{\bvarphi}\ra\ol{\bQ}_p$ as in   Proposition~\ref{prop:qexp},
\begin{equation}\label{eq:theta-dagger}
    \lambda({\bm \theta}(\bvarphi))
    =\vol(K^{H,p}_f)^{-1}\cdot \imath_{G,G'}\left(\theta^\dagger\left(\phi^p_f\otimes \prod\nolimits_{\fp\mid p}\phi_{\uchi_\fp}\otimes\prod\nolimits_{\sigma:F\ra\bR} \phi^\sG_\sigma,\lambda(\bvarphi)\right) \right),
\end{equation}
where $\theta^\dagger$ denotes the theta lift with a modification at $p$ defined as in \eqref{eq:thetadagger}.
\end{thm}

\begin{proof}
Taking the $\beta$-th coefficient at the  cusp indexed by $g^p_f$ gives rise to an element in $ \pzV^{G'}_{\Pord}\big(K^{G',p}_f,\cO_L\big)^*$, inducing an $\Lambda_P$-linear map:
\[
    \pzM^{G',n-m}_{\Pord}(K^{G',p}_f,\cO_L)\lra \Lambda_P.
\]
Putting such maps together, with $\beta$ varying in $\Her_n(K)_{\geq 0}$ and $g^p_f$ varying in a set $\fC$ of representatives of $p$-aidc cusp labels inside the ordinary locus of tame level $K^{G',p}_f$, gives the $q$-expansion map
\begin{equation}\label{eq:qexp}
   \pzM^{G',n-m}_{\Pord}(K^{G',p}_f,\cO_L)\lra \bigoplus_{\fC} \Lambda_P\llb\Her_n(K)_{\geq 0}\rrb,
\end{equation}
which is an embedding by \cite[Section~8.4]{HidaPAF}. Via this embedding, we view $\pzM^{G',n-m}_{\Pord}(K^{G',p}_f,\cO_L)$ as a sub-$\Lambda_P$-module of $\bigoplus_{\fC} \Lambda_P\llb\Her_n(K)_{\geq 0}\rrb$. We consider another sub-$\Lambda_P$-module of $\bigoplus_{\fC} \Lambda_P\llb\Her_n(K)_{\geq 0}\rrb$, which we denote by $\pzM^{G',n-m}_{\Pord,\qexp}(K^{G',p}_f,\cO_L)$, consisting of elements whose specializations at $\cP_{\tau}$ belongs to the image of \eqref{eq:Ptau-embd} for almost all finite-order $\tau:T_P(\bZ_p)^\circ\ra\ol{\bQ}^\times_p$. It is clear that, viewed as  sub-$\Lambda_P$-module of $\bigoplus_{\fC} \Lambda_P\llb\Her_n(K)_{\geq 0}\rrb$,
\begin{equation}\label{eq:incl}
    \pzM^{G',n-m}_{\Pord}(K^{G',p}_f,\cO_L)
    \subset \pzM^{G',n-m}_{\Pord,\qexp}(K^{G',p}_f,\cO_L).
\end{equation}

The $q$-expansions  in Proposition~\ref{prop:qexp} gives rise to an element 
\[
    {\bm \theta}(\bvarphi)_{\qexp}\in \pzM^{G',n-m}_{\Pord,\qexp}(K^{G',p}_f,\cO_L)\otimes_{\Lambda_P}\bI_{\bvarphi}
\]
such that for all $\lambda$ in the statement of the proposition $\lambda({\bm \theta}(\bvarphi)_{\qexp})$ equals the $q$-expansion of the right hand side of \eqref{eq:theta-dagger}. Therefore, it suffices to show that the inclusion \eqref{eq:incl} is an equality. The argument is standard in Hida theory. We include it here for the convenience for the reader.

First, we show that \begin{equation}\label{eq:incl}
    \pzM^{G',n-m}_{\Pord}(K^{G',p}_f,\cO_L)\otimes_{\Lambda_P}{\rm Frac}(\Lambda_P)
    = \pzM^{G',n-m}_{\Pord,\qexp}(K^{G',p}_f,\cO_L)\otimes_{\Lambda_P}{\rm Frac}(\Lambda_P).
\end{equation}
Take a basis $\cF_1,\dots,\cF_d$ of the left hand side (as a ${\rm Frac}(\Lambda_P)$-vector space), and extend it to a basis of the right hand side. Then we can write $\cF=a_1\cF_1+\cdots+a_d\cF_d+\cF'$ with $\cF'=0$ or an element inside the complement of the left hand side inside the right hand side. For infinitely many $\underline{\tau}:P/SP(\bZ_p)^\circ\ra\ol{\bQ}^\times_p$, the specialization of $\cF$ at $\cP_{\underline{\tau}}$ belongs to the specialization of $\pzM^{G',n-m}_{\Pord}(K^{G',p}_f,\cO_L)$, which is spanned by the specializations of $\cF_1,\dots,\cF_d$ over $\Lambda_P/\cP_{\underline{\tau}}$. It follows that $\cF'\equiv 0\mod \cP_{\underline{\tau}}$ for infinitely many $\underline{\tau}$, which implies that $\cF'=0$. This proves \eqref{eq:incl}. To finish the proof, it remains to showing that 
\begin{equation}\label{eq:Mqexp}
    \pzM^{G',n-m}_{\Pord}(K^{G',p}_f,\cO_L)=\pzM^{G',n-m}_{\Pord}(K^{G',p}_f,\cO_L)\otimes_{\Lambda_P}{\rm Frac}(\Lambda_P)\cap \pzM^{G',n-m}_{\Pord,\qexp}(K^{G',p}_f,\cO_L).
\end{equation}
Let $\pzC\subset \pzV^{G'}_{\Pord}\big(K^{G',p}_f,\cO_L\big)^*$ be the $\Lambda_P$-submodule generated by elements obtained as taking coefficients of $q$-expansions at cusps inside the ordinary locus. Denoting by $\fm$ the maximal ideal of $\Lambda_P$ and $\ol{\pzC}$ the image of $\pzC$ inside $\pzV^{G'}_{\Pord}\big(K^{G',p}_f,\cO_L\big)^*\otimes (\Lambda_P/\fm)=\Hom_{\bZ_p}(\pzV^{G',n-m}_P(K^{G',p}_f,\cO_L)[\fm],\bQ_p/\bZ_p)$. It follows from the $q$-expansion principle that 
\[\ker(\ol{\pzC})=\{v\in \pzV^{G',n-m}_P(K^{G',p}_f,\cO_L)[\fm]:\eta(v)=0\text{ for all }\eta\in \pzC\}=0.\]
Since $\pzV^{G'}_{\Pord}\big(K^{G',p}_f,\cO_L\big)^*$ is a finite $\Lambda_P$-module, $\pzV^{G',n-m}_P(K^{G',p}_f,\cO_L)[\fm]$  is finite dimensional over $\Lambda_P/\fm$, and $\ker(\ol{\pzC})=0$ implies that $\ol{\pzC}=\pzV^{G'}_{\Pord}\big(K^{G',p}_f,\cO_L\big)^*\otimes (\Lambda_P/\fm)$. By Nakayama's Lemma, we know $\pzC=\pzV^{G'}_{\Pord}\big(K^{G',p}_f,\cO_L\big)^*$, which implies \eqref{eq:Mqexp}.

\end{proof}

\begin{rmk}
One main feature of the Hida family ${\bm \theta}(\bvarphi)$ constructed in Theorem~\ref{thm:family} lies in its explicit specialization formula \eqref{eq:theta-dagger}. Such a formula is expected to be useful for applying it to study the family of Galois representations associated to $\bvarphi$. When $\bvarphi$ is a Hecke eigen-family with associated Hecke eigensystem $\lambda_{\bvarphi}$, based on results on theta correspondence for automorphic representations, it is easy to directly write down a Hecke eigensystem $\theta(\lambda_{\bvarphi})$ (in a manner like in \cite[Section~4.1.3]{pLF-F}), and deduce that $\left(\pzM^{G',n-m}_{\Pord}(K^{G',p}_f,\cO_L)\otimes_{\Lambda_P} \bI_{\bvarphi}\right)[\theta(\lambda_{\bvarphi})]$ has rank $\geq 1$ over $\bI_{\bvarphi}$. Any element in this family can be viewed as a theta lift of the family $\bvarphi$, but it might be hard to pin down its specializations without a scalar ambiguity.

\end{rmk}


\bibliographystyle{halpha}
\bibliography{BiStd}

\end{document}